\documentclass{amsart}
\usepackage[nohug]{diagrams}
\usepackage{amscd}
\usepackage[mathscr]{eucal}
\usepackage{amssymb, amsmath, amsthm}
\usepackage{amsfonts}
\usepackage{mathtools}
\usepackage{latexsym}
\usepackage{tabularx,array}
\usepackage{hyperref}
\usepackage{latexsym}
\usepackage{enumerate}

\newfont{\msam}{msam10}

\newtheorem{theorem}[]{Theorem}
\newtheorem{proposition}[]{Proposition}
\newtheorem{corollary}[]{Corollary}
\newtheorem{lemma}[]{Lemma}

\theoremstyle{definition}
\newtheorem{definition}[]{Definition}
\newtheorem{remark}[]{Remark}
\newtheorem{example}[]{Example}

\let\nc\newcommand

\def\bthm{\begin{theorem}}
	\def\ethm{\end{theorem}}
\def\blemma{\begin{lemma}}
	\def\elemma{\end{lemma}}
\def\bproof{\begin{proof}}
	\def\eproof{\end{proof}}
\def\bprop{\begin{proposition}}
	\def\eprop{\end{proposition}}
\def\bcor{\begin{corollary}}
	\def\ecor{\end{corollary}}
\nc{\la}{\label}
\def\GrCoAlg{\mathtt{GrCoAlg}}
\def\GrCoLAlg{\mathtt{GrLieCoAlg}}
\def\GrCHAlg{\mathtt{GrCHopfAlg}}

\def\CDGC{\mathtt{CDGC}}

\def\CTw{\mathtt{CTw}}

\def\DGA{\mathtt{DGA}}

\def\CDGA{\mathtt{CDGA}}

\def\DGMod{\mathtt{DG\,Mod}}

\def\DGBimod{\mathtt{DG\,Bimod}}

\def\mfa{\mathfrak{a}}
\nc{\Ob}{{\rm Ob}}
\nc{\Hom}{{\rm{Hom}}}
\nc{\bHom}{{\mathbf{Hom}}}
\nc{\bDer}{{\mathbf{Der}}}
\nc{\Homcont}{{\mathcal{H}om}}
\nc{\HOM}{\underline{\rm{Hom}}}
\nc{\DER}{\underline{\rm{Der}}}
\nc{\END}{\underline{\rm{End}}}
\nc{\bSym}{\mathbf{Sym}}
\nc{\Ext}{{\rm{Ext}}}
\nc{\Rep}{{\rm{Rep}}}
\nc{\DRep}{{\rm{DRep}}}
\nc{\NCRep}{\widetilde{\rm{Rep}}}
\nc{\RAct}{{\rm{RAct}}}
\nc{\bs}{\backslash}
\nc{\ob}{{\tt{Obs}}}
\nc{\CE}{\mathcal{C}}
\nc{\TP}{{T\!P}}

\nc{\nn}{{{\natural} {\natural}}}
\nc{\n}{{{\natural}}}
\nc{\A}{\mathbb A}
\nc{\B}{{\mathrm{B}}}
\nc{\Ba}{\overline{\mathrm{B}}}
\nc{\bC}{\overline{C}}
\nc{\bOmega}{\boldsymbol{\Omega}}
\nc{\bB}{\boldsymbol{\mathrm{B}}}
\nc{\rbB}{\overline{\boldsymbol{\mathrm{B}}}}
\nc{\EXT}{\underline{\rm{Ext}}}
\nc{\TOR}{\underline{\rm{Tor}}}
\def\H{\mathrm H}

\def\HC{\mathrm{HC}}

\def\rHC{\overline{\mathrm{HC}}}
\def\rHH{\overline{\mathrm{HH}}}

\nc{\End}{{\rm{End}}}
\nc{\GL}{{\rm{GL}}}
\nc{\gl}{{\mathfrak{gl}}}
\nc{\rgl}{\overline{{\mathfrak{gl}}}}
\nc{\g}{{\mathfrak{g}}}
\nc{\h}{{\mathfrak{h}}}
\nc{\PGL}{{\rm{PGL}}}
\nc{\SL}{{\rm{SL}}}
\nc{\sll}{\mathfrak{sl}}
\nc{\cn}{ \mbox{\rm c\^{o}ne} }
\nc{\PSL}{{\rm{PSL}}}
\nc{\ad}{{\rm{ad}}}
\nc{\Ad}{{\rm{Ad}}}
\nc{\dlim}{\varinjlim}
\nc{\plim}{\varprojlim}
\nc{\colim}{{\tt{colim}}}

\newcommand{\HH}{{\rm{HH}}}

\newcommand{\Tot}{{\rm{Tot}}}

\newcommand{\Sym}{{\rm{Sym}}}

\newcommand{\id}{{\rm{id}}}

\newcommand{\Ker}{{\rm{Ker}}}

\newcommand{\Coker}{{\rm{Coker}}}

\def\cb{\boldsymbol{\Omega}}

\def\bs{\backslash}

\def\U{\mathcal{U}}

\nc{\env}{\mathrm{End}(V)}
\nc{\FT}{\mathcal{C}}

\numberwithin{equation}{section}
\numberwithin{theorem}{section}
\numberwithin{lemma}{section}
\numberwithin{proposition}{section}
\numberwithin{corollary}{section}
\numberwithin{example}{section}
\numberwithin{remark}{section}

\newcommand{\dual}{\text{!`}}

\nc{\Char}{{\rm{Ch}}}
\nc{\mfg}{\mathfrak{g}}
\nc{\mfG}{\mathfrak{G}}
\nc{\CS}{{\rm{CS}}}
\nc{\bS}{\mathbb{S}}
\nc{\cO}{\mathcal{O}}
\nc{\RHom}{\rm{RHom}}
\nc{\LL}{\mathcal{L}}
\nc{\Id}{\mathrm{Id}}
\nc{\defeq}{\vcentcolon=}
\nc{\eqdef}{=\vcentcolon}
\nc{\cone}{{\rm{cone}\,}}

\begin{document}
\title{Curved Koszul duality and cyclic (co)homology}

\author{Yining Zhang}
\address{Department of Mathematics,
University of Colorado Boulder,
Boulder, CO 80309, USA}
\email{Yining.Zhang@colorado.edu}

\begin{abstract}
We study the curved Koszul duality theory for associative algebras presented by quadratic-linear-constant (QLC) relations. As an application, we investigate the cyclic (co)homology of a QLC algebra and its Koszul dual curved DG algebra, and extend a result due to Feigin and Tsygan. We also study a commutative/Lie analog of this curved Koszul duality theory.
\end{abstract}
\maketitle

\section{Introduction}
Hochschild (co)homology of algebras was introduced by Hochschild \cite{Hoch45}. Later, Cartan and Eilenberg \cite{CE56} provided an equivalent definition in terms of derived functors. The canonical tool for computing Hochschild homology is the {\it bar resolution}. In some instances, however, one is able to find a simpler resolution. For example, let $\Sym(V)$ be the symmetric algebra over a vector space $V$. It turns out that the free $\Sym(V)$-bimodule generated by the exterior algebra over $V$, 
$$\Sym(V)\otimes \wedge V \otimes \Sym(V)\,,$$
with an appropriate differential, is a resolution of $\Sym(V)$ in the category of $\Sym(V)$-bimodules. For the tensor algebra $T(V)$, the following length $2$ complex can be used to compute the Hochschild homology of $T(V)$:
\begin{equation}\label{TVRes}
\cdots \longrightarrow 0 \longrightarrow T(V)\otimes V\otimes T(V) \longrightarrow T(V)\otimes T(V)\,,
\end{equation}
where the only nontrivial differential is given by $p\otimes v\otimes q\mapsto p\otimes vq -pv\otimes q$. 
An associative algebra presented by quadratic relations ({\it quadratic algebra} for short) is an algebra of the form $T(V)/(R)$, where $V$ is a vector space and the two-sided ideal $(R)$ is generated by $R\subset V^{\otimes 2}$. Both symmetric algebra and tensor algebra are examples of quadratic algebras. The Koszul duality theory for quadratic algebras was introduced by Priddy in \cite{Priddy}. In more detail, given a quadratic algebra $A=T(V)/(R)$, the quadratic data $(V,\,R)$ permits one to construct a graded coalgebra $A^\dual$ called the {\it Koszul dual coalgebra} of $A$,  and a degree $-1$ linear map $\kappa$ from $A^\dual$ to $A$ satisfying so called the {\it Maurer-Cartan equation}. $\kappa$ provides a differential on the graded free bimodule $A\otimes A^\dual \otimes A$. We denote the resulting complex by $A\otimes_\kappa A^\dual \otimes_\kappa A$, and call it the {\it total Koszul complex}. $A$ is called {\it Koszul} if the morphism from $A\otimes_\kappa A^\dual \otimes_\kappa A$ to $A$, induced by the the counit on $A^\dual$ and the multiplication on $A$, is a quasi-isomorphism. In this case, the total Koszul complex $A\otimes_\kappa A^\dual \otimes_\kappa A$ is called the {\it Koszul resolution} of the quadratic algebra $A$. If $A=\Sym(V)$, it can be checked that the underlying vector space of $A^\dual$ is isomorphic to the exterior algebra $\wedge V$. Moreover, if $A=T(V)$, then the Koszul dual coalgebra $A^\dual=k\oplus V[1]$, and the total Koszul complex $A\otimes_\kappa A^\dual \otimes_\kappa A$ is precisely \eqref{TVRes}.

The quadratic hypothesis $R\subset V^{\otimes 2}$ can be weakened by only requiring $R \subset V\oplus V^{\otimes 2}$. In this case, the algebra $A=T(V)/(R)$ is called a {\it quadratic-linear algebra} (QL algebra for short). The Koszul duality theory for QL algebras is also discussed in \cite{Priddy}. Let $q\,:\,T(V)\twoheadrightarrow  V^{\otimes 2}$ be the projection onto the quadratic part of tensor algebra, and $qR$ be the image of $R$ under $q$. $T(V)/(qR)$ is a quadratic algebra, denoted by $qA$, and is called the {\it associated quadratic algebra} of the QL algebra $A$. The linear terms in $R$ induce a differential on $(qA)^\dual$. Given an augmented associative algebra $A=k\oplus \bar{A}$, it can be viewed as a QL algebra by the isomorphism
$$A\,\cong\,T(\bar{A})/(R)\,,$$
where $R$ is generated by $a\otimes a' -aa'$ for $a,\,a'\in \bar{A}$. Then $(qA)^\dual$ equipped with the differential described above is nothing but the {\it bar construction} $\bB(A)$ of $A$. Similar to the case of quadratic algebras, one can construct a degree $-1$ map $\kappa$ from $(qA)^\dual$ to $A$ satisfying the Maurer-Cartan equation, and it would induce a differential on the complex $A\otimes_\kappa (qA)^\dual \otimes_\kappa A$. It turns out that if $qA$ is Koszul, there is a quasi-isomorphism $A\otimes_\kappa (qA)^\dual \otimes_\kappa A \xrightarrow{\sim} A$. The universal enveloping algebra $A=\U\g$ of a finite-dimensional Lie algebra $\g$ is a QL algebra, and $(qA)^\dual$ is the Chevalley-Eilenberg DG coalgebra of $\g$. Using the resolution $A\otimes_\kappa (qA)^\dual \otimes_\kappa A$, one could identify the Hochschild (co)homology of $\U\g$ with the Lie algebra (co)homology of $\g$ with coefficients in $\U\g$.

It is natural to further weaken the hypothesis and add constant terms to the space of relations $R$. In this case, $(V,\,R)$ is called a {\it quadratic-linear-constant} data, and $A=T(V)/(R)$ is called a quadratic-linear-constant algebra (QLC algebra for short). The Koszul dual (curved) algebra $(qA)^!$ of a QLC algebra $A$ was studied in \cite{Pos93, PP05}. In this paper, we study the the Koszul dual (curved) coalgebra $(qA)^\dual$ of $A$, and apply the curved Koszul duality developed in \cite {Pos11, HM12} to QLC algebras. In Section \ref{secPreliminary}, we review the Koszul duality theory for curved DG algebras and coalgebra and present some preliminary results. In particular, we show that given a linear map $\alpha$ of degree $-1$ from a curved DG coalgebra $C$ to a DG algebra $A$, if $\alpha$ is a {\it curved twisting morphism} (see Section \ref{subsecCTW} for the definition), then $\alpha$ provides a differential on the graded free bimodule $A\otimes C \otimes A$, and we denote the resulting complex by $A\otimes_\alpha C \otimes_\alpha A$ (see Proposition \ref{TwTensorBimod}).

In Section \ref{secQLC}, we show that the linear and constant terms in $R$ equip the Koszul dual coalgebra $(qA)^\dual$ the structure of a curved DG coalgebra (see Proposition \ref{KoszulCurCAg}). Moreover, there is a curved twisting morphism $\kappa\,:\,(qA)^\dual\rightarrow A$. To any coaugmented curved DG algebra $C$, one can associate a semi-free DG algebra $\cb(C)$, called the {\it cobar construction} of $C$. It turns out that if the quadratic algebra $qA$ is Koszul, $\kappa$ induces a quasi-isomorphism from $\cb((qA)^\dual)$ to $A$ (see Theorem \ref{KoszulFreeRes}). The {\it total Koszul complex} of a QLC algebra $A$ is the complex $A\otimes_\kappa(qA)^\dual\otimes_\kappa A$. Our first main result generalizes the Koszul resolution of quadratic algebras to QLC algebras.
\begin{theorem}\label{IntroKoszulRes}(see Theorem \ref{KoszulRes})
Let $A$ be a QLC algebra, and $(qA)^\dual$ be its Koszul dual curved coalgebra. Suppose $qA$ is Koszul, then the total Koszul complex $A\otimes_\kappa(qA)^\dual\otimes_\kappa A$ is a resolution of $A$ in the category of $A$-bimodules. 
\end{theorem}
As applications of Theorem \ref{IntroKoszulRes}, we investigate the Hochschild homologies of the Weyl algebra and the universal enveloping algebra of a ``unital'' Lie algebra, and recover a result due to Sridharan \cite{Sridharan61}.

In Section \ref{secCycHo}, we turn to the cyclic homology of a QLC algebra. The cyclic (co)homology $\HC(A)$ of an algebra $A$ was introduced independently by Tsygan and Connes. Cyclic cohomology was developed as a noncommutative generalization of the de Rham cohomology \cite{Connes85}, and cyclic homology can be considered as an additive analog of algebraic K-theory \cite{Tsygan83, LQ}. There are two definitions of cyclic homology: one in terms of the standard bicomplex (see, for example, \cite{L}) and another one in terms of the (non-abelian) derived functor in the sense of Quillen's homotopical algebra \cite{Qui67} (see Theorem \ref{FTrHC}). By modifying the cyclic bicomplex, one obtain two variations of cyclic homology:  periodic cyclic homology and negative cyclic homology, denoted by $HC^{\mathrm{per}}$ and $HC^{-}$, respectively. Moreover, those two variations are related to the equivariant cohomology of the free loop space (see, for example, \cite{Good85, Jones87}). Feigin and Tsygan studied the (reduced) cyclic (co)homology of a QL algebra in \cite{FT87}, and they obtained the following long exact sequence provided that $qA$ is Koszul:
\begin{equation}\label{IntroFTHCLongExactSeq}
\cdots\,\longrightarrow\,\rHC_{\mathrm{per}}^{-2-n}((qA)^!)\,\longrightarrow\,\rHC_{n}(A)\,\longrightarrow\,\rHC^{-1-n}((qA)^!)\,\longrightarrow\,\rHC_{\mathrm{per}}^{-1-n}((qA)^!)\,\longrightarrow\,\cdots\,.
\end{equation}
We extend this result to the case of Koszul QLC algebras in Section \ref{secCycHo} and obtain our second main result
\begin{theorem}\label{IntroThmFTHCLongExactSeq}(see Theorem \ref{ThmFTHCLongExactSeq})
Let $A=A(V,\,R)$ be a QLC algebra, where $V$ is a non-negatively graded locally finite graded vector space, and $(qA)^!$ be its Koszul dual curved DG algebra. Suppose $qA$ is Koszul, then one has the long exact sequence \eqref{IntroFTHCLongExactSeq}.
\end{theorem}
The first part of Section \ref{secCycHo} is devoted to the constructions of cyclic (co)homology of curved DG algebras. Our approach is extending the cyclic bicomplex to curved DG algebras instead of the mixed complex as, for example, in \cite{BW20}. There is a long exact sequence connecting the three cyclic theories of a curved DG algebra $\widetilde{A}$:
\begin{equation}\label{IntrocoHCLongExactSeq}
\cdots\,\longrightarrow\,\rHC_{\mathrm{per}}^{-2-n}(\widetilde{A})\,\longrightarrow\,\rHC_{-}^{-n}(\widetilde{A})\,\longrightarrow\,\rHC^{-1-n}(\widetilde{A})\,\longrightarrow\,\rHC_{\mathrm{per}}^{-1-n}(\widetilde{A})\,\longrightarrow\,\cdots\,.
\end{equation}
In Section \ref{subsecNCDiff}, we review the bimodule $\Omega^1R$ of noncommutative differential 1-forms of a DG algebra $R$. In particular, there is a periodic complex
\begin{equation}\label{IntroXPlusComplex}
X^+(R)\defeq[0 \longleftarrow \bar{R} \xleftarrow{\ \beta\ } \Omega^{1}R_{\n} \xleftarrow{\ \bar{\partial}\ } \bar{R} \xleftarrow{\ \beta\ }  \Omega^{1}R_{\n} \xleftarrow{\ \bar{\partial}\ } \cdots]\,,
\end{equation}
where $\bar{R}\defeq R/k$ and $\Omega^{1}R_{\n}$ is the commutator quotient space of $\Omega^1R$. Let $R=\cb((qA)^\dual)$ be the semi-free resolution of $A$ constructed in Section \ref{subsecCobarKDCCA}. It is shown in \cite{BKR} that the homology of the total complex of $X^+(R)$ is isomorphic to the reduced cyclic homology of $A$. On the other hand, we show in Section \ref{subsecFTHCLongExactSeq} that $X^+(R)$ is isomorphic to the reduce negative cyclic cochain bicomplex of $(qA)^!$. Thus, one has the following isomorphism
\begin{eqnarray}\label{IntroHCiso}
\rHC_n(A)\,\cong\,\rHC_{-}^{-n}((qA)^!)\,.
\end{eqnarray}
Theorem \ref{IntroFTHCLongExactSeq} follows from the long exact sequence \eqref{IntrocoHCLongExactSeq} and the isomorphism \eqref{IntroHCiso}.

In Section \ref{secComQLC}, we study the Koszul duality theory for commutative QLC algebras which can be viewed as a commutative/Lie analog to Section \ref{secQLC}. A commutative QLC algebra $A$ is of the form $\Sym(V)/(R)$ where $V$ is a vector space and the ideal $(R)$ is generated by a subspace $R \subset k\oplus V\oplus \Sym^2(V)$. Let $q\,:\,\Sym(V)\twoheadrightarrow  \Sym^2(V)$ be the projection onto the quadratic part of symmetric algebra, and $qR$ be the image of $R$ under $q$. Then $qA\defeq\overline{\Sym(V)}/(qR)$, where $\overline{\Sym(V)}\defeq \Sym(V)/k$, is a nonunital commutative quadratic algebra. The Koszul duality theory for quadratic algebras over binary quadratic operads was introduced by Ginzburg and Kapranov in \cite{GK94}. It was generalized to monogenic algebras over quadratic operads by Mill\`{e}s in \cite{JBMilles12}. In our case, the Koszul dual coalgebra $(qA)^\dual$ of $qA$ is a curved DG Lie coalgebra. For a DG Lie coalgebra $\mfG$, the {\it Lie cobar construction} $\cb_{\mathtt{Lie}}(\mfG)$ is linearly dual to the Lie bar construction produced by Quillen in \cite{Qui69}. Moreover, this construction can be extended to the category of curved DG Lie coalgebras. We showed that if $qA$ is Koszul in the sense of \cite{JBMilles12}, then $\cb_{\mathtt{Lie}}((qA)^\dual)$ is a semi-free resolution of $A$ in the category of commutative DG algebras (see Theorem \ref{CKoszulFreeRes}). It is clear that every commutative QLC algebra can be viewed as an (associative) QLC algebra. Hence, there are two coalgebras associated to $A$: the Koszul dual Lie coalgebra $(qA)^\dual$, and the Koszul dual (coassociative) coalgebra (of $A$ as an associative algebra), denoted by $(qA^+)^\dual$. Our final result is the following theorem explaining the relation between $(qA)^\dual$ and $(qA^+)^\dual$
\begin{theorem}\label{IntroThmKDLiecoalgKDcoalg}(see Theorem \ref{ThmKDLiecoalgKDcoalg})
There is a natural isomorphism
$$\U^c((qA)^\dual)\,\cong\,(qA^+)^\dual\,,$$
where $\U^c(\mbox{--})$ denotes universal coenveloping coalgebra functor from the category of graded Lie coalgebras to the category of graded coassociative coalgebras.
\end{theorem}

\subsection*{Acknowledgements} 
The author is grateful to Alexander Gorokhovsky for introducing the work of Feigin and Tsygan \cite{FT87} and reading through the preliminary version of this paper. The author would like to thank Ajay C. Ramadoss for interesting discussions and comments.

\section{Koszul duality for curved DG (co)algebras}\label{secPreliminary}
In this section, we recall the notion of curved DG (co)algebras, their (co)bar constructions, and curved twisiting morphisms. We refer to the book of Loday and Vallette \cite{LV} for a complete exposition about Koszul duality theory for (uncurved) DG (co)algebras, and the monograph of Positselski \cite{Pos11} for a detailed treatment on curved DG (co)algebras and their (co)modules. See \cite{HM12} for generalizations to operads and properads, and see \cite{Lyuba13} for generalizations to curved $A_\infty$ (co)algebras.
\subsection{Notations and conventions}
Throughout this paper, $k$ denotes a field of characteristic zero. Let $V$ be a graded $k$-vector space, the {\it suspension} of $V$ is the graded vector space $sV$ such that $(sV)_{i}=V_{i-1}$. Similarly, the {\it desuspension} of $V$ is the graded vector space $s^{-1}V$ such that $(s^{-1}V)_{i}=V_{i+1}$. Notice that we will use $V[1]$ (resp., $V[-1]$) to denote $sV$(resp., $s^{-1}V$) as well. Let $V$ and $W$ be two graded $k$-vector spaces, the space of $k$-morphisms of degree zero is denoted by $\Hom_{k}(V,\,W)$. The graded vector space of $k$-morphisms of any degree is denoted by $\bHom_{k}(V,\,W)$, and one has
$$\bHom_{k}(V,\,W)\,=\,\bigoplus_{p}\Hom_{k}(V,\,W[-p])\,=\,\bigoplus_{p}\big(\prod_{i} \Hom_{k}(V_{i},\,W_{i+p})\big)\,.$$
The graded linear dual $V^\ast$ of a graded vector space $V$ is $\bHom_{k}(V,\,k)$, i.e.,
$$V^\ast\,=\,\bigoplus_{p}\Hom_{k}(V_{-p},\,k)\,.$$

An algebra means an associative unital algebra over $k$. The associative multiplication is denoted by $\mu\,:\,A\otimes A \,\rightarrow\, A$, and the unit is denoted by $u\,:\,k\,\rightarrow\, A$. We denote by $1_{A}$, or simply by $1$, the image of $1_{k}$ in $A$ under $u$. An algebra is {\it augmented} if it is given together with a morphism of algebras $\epsilon\,:\,A\,\rightarrow\,k$, called the {\it augmentation map}. In particular, $\epsilon(1_{A})=1_{k}$. If $A$ is augmented, one has $A=k\oplus \bar{A}$, where $\bar{A}=\ker(\epsilon)$ called the {\it augmentation ideal} of $A$.

Unless stated otherwise, all {\it differential graded (DG)} objects are equipped with differentials of degree $-1$. Given a DG algebra $A$, let $A^{op}$ denote its opposite DG algebra. The {\it enveloping algebra} of $A$ is $A^{e}\defeq A\otimes A^{op}$. The category $\DGBimod(A)$ of $A$-DG bimodules is equivalent to the category $\DGMod$-$A^{e}$ of right $A^{e}$-DG modules as well as the category $A^{e}$-$\DGMod$ of left $A^{e}$-DG modules. 

A coalgebra means a coassociative counital coalgebra over $k$. The coassociative comultiplication is denoted by $\Delta\,:\,C\,\rightarrow\,C\otimes C$, and the counit is denoted by $\epsilon\,:\,C\,\rightarrow\,k$. A coalgebra is {\it coaugmented} if it is given together with a morphism of coalgebras $\eta\,:\,k\,\rightarrow\, C$, called the {\it coaugmentation map}, such that $\epsilon \circ \eta=\id_{k}$. If $C$ is coaugmented, then $C=k\oplus\bar{C}$ where $\bar{C}$ is the cokernel of $\eta$. We shall often work with coaugmented coassociative counital DG coalgebras $C$ which are {\it conilpotent} in the sense that
\begin{equation*}
\bar{C}\,=\,\bigcup_{n\geqslant 2}\ker[C\xrightarrow{\Delta^{(n)}}C^{n}\twoheadrightarrow \bar{C}^{n}]\,,
\end{equation*}
where $ \Delta^{(n)} $ denotes the $n$-th iteration of the comultiplication map $\Delta$. In particular, we shall call $C\xrightarrow{\Delta} C^{2}\twoheadrightarrow \bar{C}^{2}$ the {\it reduced comultiplication}, and denote it by $\bar{\Delta}$. Notice that $\bar{C}$ is a coassociative non-counital coalgebra with comultiplication $\bar{\Delta}$.

\subsection{Curved DG (co)algebra}
A {\it curved DG algebra} over $k$ is a $k$-graded algebra $A$ endowed with a derivation $\triangledown$ of degree $-1$ and an element $\Theta$ of degree $-2$ satisfying
\begin{enumerate}
	\item $\triangledown^2\,=\, [\Theta,\,\mbox{--}]$
	\item $\triangledown(\Theta)\,=\,0$
\end{enumerate}

A {\it curved DG coalgebra} over $k$ is a $k$-graded coalgebra $C$ endowed with a coderivation $d$ of degree $-1$ and a linear map $h\,:\,C\rightarrow k$ of degree $-2$ satisfying
\begin{enumerate}
	\item $d^2\,=\, (h\otimes \id- \id\otimes h)\circ\Delta$
	\item $h\circ d\,=\,0$
\end{enumerate}

\begin{proposition}\label{CurvedDGAC}
Let $(A,\,\triangledown,\,\Theta)$ be a non-positively graded, locally finite curved DG algebra over $k$, then its graded linear dual $C\defeq A^\ast$ is a curved DG coalgebra.
\end{proposition}

\begin{proof}
By the assumptions, $C\defeq A^\ast$ is a graded coalgebra over $k$. And it is easy to see that $d\defeq \triangledown^\ast$ is a coderivation of degree $-1$ on $C$. Let $h=-\Theta$ viewed as a linear functional of degree $-2$ on $C$.
Therefore, for any $f\in C$, 
$$(d^2)(f)\,=\,-f\big(\triangledown^2(\mbox{--})\big)\,=\,f\big([-\Theta,\,\mbox{--}]\big)\,=\,(h\otimes \id- \id\otimes h)\Delta(f)\,,$$
and,
$$h\big(d(f)\big)\,=\,(-1)^{|f|}f\big(\triangledown(\Theta)\big)\,=\,0\,.$$
\end{proof}

\begin{remark}
The converse of the previous proposition is also true. Namely, the graded linear dual of a non-negatively graded locally finite curved DG coalgebra is a curved DG algebra.
\end{remark}

A {\it strict morphism} $f$ between two curved DG algebras $(A,\,\triangledown,\,\Theta)$ and $(A',\,\triangledown',\,\Theta')$ is a morphism of graded algebra $f\,:\,A\rightarrow A'$ such that $\triangledown'\circ f=f\circ \triangledown$ and $\Theta'=f(\Theta)$. Let $\CDGA_k$ denote the category of curved DG algebras over $k$ whose morphisms are strict morphisms between curved DG algebras. The category of DG algebras $\DGA_k$ is a full category of  $\CDGA_k$.

A {\it strict morphism} $f$ between two curved DG coalgebras $(C,\,d,\,h)$ and $(C',\,d',\,h')$ is a morphism of graded coalgebra $f\,:\,C\rightarrow C'$ such that $d'\circ f=f\circ d$ and $h'\circ f=h$. Let $\CDGC_k$ denote the category whose objects are curved DG coalgebras over $k$ such that their underlying graded coalgebras are conilpotent, and morphisms are strict morphisms between curved DG coalgebras. 

\begin{remark}
There is a more general notion of morphisms between curved DG (co)algebras, see \cite{Pos11} for the details.
\end{remark}

\subsection{Cobar constructions of curved DG coalgebras}\label{subsecCobar}
Recall from \cite{Pos11} that a  curved DG coalgebra $(C,\,d,\,h)$ is called {\it coaugmented} if it is coaugmented as a graded coalgebra, and the coaugmentation map $\eta\,:\,k\rightarrow C$ satisfies $d\circ \eta=0$ and $h\circ \eta=0$. Notice that if $C$ is non-negatively graded, then these two conditions are automatically satisfied. In addition, it can be shown that $d$ and $h$ induce the structure of a curved DG coalgebra on $(\bar{C},\,\bar{\Delta})$.

The {\it cobar construction} $\cb(C)$ of a coaugmented curved DG coalgebra $(C,\,d,\,h)$ is defined as the graded tensor algebra $T(\bar{C}[-1])$ equipped with three derivations $d_{0}$, $d_{1}$, and $d_{2}$, where $d_{0}$ (resp., $d_{1}$) is induced by $h$ (resp., $d$) on $\bar{C}$ and $d_{2}|_{\bar{C}[-1]}$ is given by the composite map
$$\bar{C}[-1] \cong k[-1] \otimes \bar{C}\xrightarrow{\Delta_{-1} \otimes \bar{\Delta}} k[-1] \otimes k[-1] \otimes \bar{C}\otimes \bar{C}\cong \bar{C}[-1]\otimes \bar{C}[-1]\,.$$
Here, $\Delta_{-1}\,:\,k[-1] \rightarrow k[-1] \otimes k[-1]$ is the map of degree $-1$ taking $1_{k[-1]}$ to $-1_{k[-1]} \otimes 1_{k[-1]}$. More precisely, given $(s^{-1}c_{1},\,\cdots,\,s^{-1}c_{n})\in \cb(C)$, 
\begin{equation}\label{CobarD0}
d_{0}(s^{-1}c_{1},\,\cdots,\,s^{-1}c_{n})\,=\,\sum_{i=1}^{ n}(-1)^{|c_{1}|+\cdots+|c_{i-1}|+i-1}h(c_{i})(s^{-1}c_{1},\,\cdots,\, s^{-1}c_{i-1},\,s^{-1}c_{i+1},\,\cdots,\,s^{-1}c_{n})\,,
\end{equation}
\begin{equation}\label{CobarD1}
d_{1}(s^{-1}c_{1},\,\cdots,\,s^{-1}c_{n})\,=\,\sum_{i=1}^{ n}(-1)^{|c_{1}|+\cdots+|c_{i-1}|+i}(s^{-1}c_{1},\,\cdots,\,s^{-1}dc_{i},\,\cdots,\,s^{-1}c_{n})\,,
\end{equation}
and
\begin{equation}\label{CobarD2}
d_{2}(s^{-1}c_{1},\,\cdots,\,s^{-1}c_{n})\,=\,\sum_{i=1}^{n}(-1)^{|c_{1}|+\cdots+|c_{i-1}|+|c_{i}'|+i}(s^{-1}c_{1},\,\cdots,\,s^{-1}c_{i}^{\prime},\,s^{-1}c_{i}^{\prime\prime},\,\cdots,\,s^{-1}c_{n})\,.
\end{equation}
where $\bar{\Delta}(c_{i})=c_{i}'\otimes c_{i}''$ in the Sweedler notation.

\begin{proposition}\label{CobarDGA}
Given a coaugmented curved DG coalgebra $(C,\,d,\,h)$, the cobar construction $\cb(C)$ is a DG algebra with differential given by $d_0+d_1+d_2$.
\end{proposition}

\begin{proof}
Since $d_0$, $d_1$, and $d_2$ are derivations of degree $-1$, $(d_0+d_1+d_2)^2=\frac{1}{2}[d_0+d_1+d_2,\,d_0+d_1+d_2]$. It suffices to show that for any $s^{-1}c\in \bar{C}[-1]$,  
\begin{eqnarray*}
&& (d_0+d_1+d_2)^2(s^{-1}c)\\
&=& (d_0d_0+d_0d_1+d_0d_2+d_1d_0+d_1d_1+d_1d_2+d_2d_0+d_2d_1+d_2d_2)(s^{-1}c)\\
&=& 0\,.
\end{eqnarray*}
Indeed,
\begin{itemize}
    \item $(d_0d_0)(s^{-1}c)=0$ and $(d_1d_0)(s^{-1}c)=0$ for the curved DG coalgebra $C$ is coaugmented;
    \item $(d_0d_1)(s^{-1}c)=0$ for $h\circ d\,=\,0$;
    \item $(d_2d_0)(s^{-1}c)=0$ by the construction of $d_2$;
    \item $(d_0d_2+d_1d_1)(s^{-1}c)=-(-1)^{|c'|}h(c')s^{-1}c''+h(c'')s^{-1}c'+s^{-1}d^2(c)=0$ since $d^2\,=\, (h\otimes \id- \id\otimes h)\circ\bar{\Delta}$;
    \item $(d_1d_2+d_2d_1)(s^{-1}c)=(-1)^{|c'|}s^{-1}d(c')\otimes s^{-1}c''-s^{-1}c'\otimes s^{-1}d(c'')+(-1)^{|d(c)'|}s^{-1}d(c)'\otimes s^{-1}d(c)''=0$ since $d$ is a coderivation;
    \item $(d_2d_2)(s^{-1}c)=0$ because of the coassociativity of $C$.
\end{itemize}
\end{proof}

Recall from \cite[Def. 2.9]{Idrissi18} (see also \cite[Sec. 3.3.1]{HM12}), a {\it semi-augmented} DG algebra $A$ is a DG algebra equipped with a linear map $\epsilon\,:\,A\,\rightarrow\,k$ (not necessarily compatible with the DG algebra structure) such that $\epsilon(1_{A})=1_{k}$. $\epsilon$ is called the {\it semi-augmentation map}. For a semi-augmented algebra $A$, let $\bar{A}$ denote the kernal of $\epsilon$. And one has an isomorphism of graded vector spaces $A\cong \bar{A}\oplus k$.

A morphism $f$ between two semi-augmented DG algebras $(A,\,d_A,\,\epsilon)$ and $(A',\,d_{A'},\,\epsilon')$ is a morphism of DG algebras $f\,:\,(A,\,d_A) \rightarrow (A',\,d_{A'})$ such that $\epsilon'\circ f=\epsilon$. We denote by $\mathtt{SemiDGA_{k/k}}$ the category of semi-augmented DG algebras over $k$. $\DGA_{k/k}$, the category of augmented DG algebras, is a full subcategory of $\mathtt{SemiDGA_{k/k}}$.

Given a coaugmented curved DG coalgebra $(C,\,d,\,h)$, the cobar construction $\cb(C)$ is semi-augmented with the canonical projection $\cb(C)\twoheadrightarrow k$ being the semi-augmentation map. Moreover, it is clear that $\cb(\mbox{--})$ is functor from the category of coaugmented curved DG coalgebras with strict morphisms to the category of semi-augmented DG algebras.

\subsection{Bar constructions of curved DG algebras}\label{subsecBar}
Recall from \cite{Pos11} that a curved DG algebra $(A,\,\triangledown,\,\Theta)$ is {\it augmented} if it is augmented as a graded algebra, and the augmentation map $\epsilon\,:\,A\,\rightarrow\,k$ satisfies $\epsilon\circ \triangledown=0$ and $\epsilon(\Theta)=0$. Notice that if $A$ is non-positively graded, then these two conditions are automatically satisfied. If $A$ is coaugmented, then $A=k\oplus\bar{A}$ as curved DG algebras, where $k$ is viewed as curved DG algebra $(k,\,0,\,0)$.

The {\it bar construction} $\bB(A)$ of an augmented curved DG algebra $(A,\,\triangledown,\,\Theta)$ is defined as the graded tensor coalgebra $T^{c}(\bar{A}[1])$ equipped with three coderivations $d_{0}$, $d_{1}$, and $d_{2}$, where $d_{0}$ (resp., $d_{1}$) is induced by $\Theta$ (resp., $\triangledown$) on $\bar{A}$ and $d_{2}$ is the coderivation whose corestriction to $\bar{A}[1]$ is given by the composite map
$$T^{c}(\bar{A}[1]) \twoheadrightarrow\bar{A}[1]\otimes \bar{A}[1]\cong k[1]\otimes k[1]\otimes \bar{A}\otimes\bar{A}\xrightarrow{\mu_{1}\otimes \mu} k[1]\otimes \bar{A}\cong\bar{A}[1]$$
where $\mu_{1}\,:\, k[1]\otimes k[1]\rightarrow k[1]$ is the map of degree $-1$ taking $1_{k[1]}\otimes 1_{k[1]}$ to $1_{k[1]}$.
More explicitly, given $(sa_{1},\,\cdots,\,sa_{n})\in \bB(A)$,
\begin{equation}\label{BarD0}
d_{0}(sa_{1},\,\cdots,\,sa_{n})\,=\,\sum_{i=0}^{ n}(-1)^{|a_{1}|+\cdots+|a_{i}|+i+1}(sa_{1},\,\cdots,\,sa_{i},\,s\Theta,\,sa_{i+1},\,\cdots,\,sa_{n})\,,
\end{equation}
\begin{equation}\label{BarD1}
d_{1}(sa_{1},\,\cdots,\,sa_{n})\,=\,\sum_{i=1}^{ n}(-1)^{|a_{1}|+\cdots+|a_{i-1}|+i}(sa_{1},\,\cdots,\,s\triangledown(a_{i}),\,\cdots,\,sa_{n})\,,
\end{equation}
and
\begin{equation}\label{BarD2}
d_{2}(sa_{1},\,\cdots,\,sa_{n})\,=\,\sum_{i=1}^{ n-1}(-1)^{|a_{1}|+\cdots+|a_{i}|+i-1}(sa_{1},\,\cdots,\,s(a_{i}a_{i+1}),\,\cdots,\,sa_{n})\,.
\end{equation}

The following proposition is formally dual to Proposition \ref{CobarDGA}, and the proof is very similar.

\begin{proposition}\label{BarDGC}
Given an augmented curved DG algebra $(A,\,\triangledown,\,\Theta)$, the bar construction $\bB(A)$ is a DG coalgebra with differential given by $d_0+d_1+d_2$.
\end{proposition}

\subsection{Curved twisting morphisms}\label{subsecCTW}
Given a curved DG algebra $(A,\,\triangledown,\,\Theta)$ with the unit map $u\,:\,k \rightarrow A$ and a curved DG coalgebra $(C,\,d,\,h)$ with the counit map $\epsilon\,:\,C \rightarrow k$, the graded vector space $\bHom_{k}(C,\,A)$ is a graded algebra with the {\it convolution product} given by 
$$f\ast g\defeq \mu\circ (f\otimes g)\circ \Delta\,:\,C\rightarrow A\,.$$
The degree $-1$ map $\partial$ on $\bHom_{k}(C,\,A)$ defined by
$$\partial(f)=\triangledown\circ f-(-1)^{|f|}f\circ d$$
is a derivation (see \cite[Prop. 2.1.1]{LV}). 

\begin{proposition}
$\big(\bHom_{k}(C,\,A),\,\partial,\,\Theta\epsilon-u\circ h\big)$ is a curved DG algebra.
\end{proposition}

\begin{proof}
For any $f\in \bHom_{k}(C,\,A)$,
\begin{eqnarray*}
\partial^2(f) &=& \triangledown\circ\partial(f)-(-1)^{|f|-1}\partial(f)\circ d\\
&=&\triangledown^2\circ f-(-1)^{|f|}\triangledown\circ f\circ d -(-1)^{|f|-1}\triangledown\circ f\circ d-f\circ d^2\\
&=&[\Theta,\,\mbox{--}]\circ f-f\circ (h\otimes \id- \id\otimes h)\circ\Delta\\
&=&(\Theta\epsilon)\ast f -f\ast (\Theta\epsilon) -(u\circ h)\ast f+f\ast (u\circ h)\\
&=&[\Theta\epsilon-u\circ h,\,f]\,.
\end{eqnarray*}
Moreover,
$$\partial(\Theta\epsilon-u\circ h)\,=\,\triangledown\circ \Theta\epsilon-\Theta\epsilon\circ d -\triangledown\circ u\circ h+u\circ h\circ d\,=\,0\,,$$
for $\triangledown(\Theta)=0$, $\epsilon\circ d=0$, $\triangledown\circ u=0$, and $ h\circ d=0$.
\end{proof}

An element $\alpha\in \bHom_{k}(C,\,A)$ of degree $-1$ is called a {\it curved twisting morphism} if it is a solution of the {\it curved Maurer-Cartan equation}
\begin{equation}\label{CMCeq}
\partial(\alpha)+\alpha\ast \alpha\,=\,-\Theta\epsilon+u\circ h\,.
\end{equation}
Write $\CTw(C,\,A)$ for the set of curved twisting morphisms from $C$ to $A$. When $A$ is (semi-)augmented and $C$ is coaugmented, a curved twisting morphism from $C$ to $A$ is supposed to send $k$ to $0$ and $\bar{C}$ to $\bar{A}$. 

\begin{proposition}\label{CobarCTwBijection}
For every semi-augmented DG alegbra $(A,\,d_A,\,\epsilon_A)$ and every coaugmented curved DG coalgebra $(C,\,d_C,\,h)$, there exists a bijection
$$\Hom_{\mathtt{SemiDGA_{k/k}}}(\cb(C),\,A)\,\cong\,\CTw(C,\,A)\,.$$
\end{proposition}

\begin{proof}
Since the underlying graded algebra of $\cb(C)$ is $T(\bar{C}[-1])$, there is a one-to-one correspondence between the set of semi-augmented algebra morphisms from $\cb(C)$ to $A$ and the set of $k$-linear maps from $\bar{C}[-1]$ to $\bar{A}$. Let $\tilde{\alpha}$ be a morphism in $\Hom_{\mathtt{SemiDGA_{k/k}}}(\cb(C),\,A)$. Its restriction to $\bar{C}[-1]$ induces a degree $-1$ map $\alpha\,:\,\bar{C}\rightarrow \bar{A}$ via $\alpha(c)=\tilde{\alpha}(s^{-1}c)$. The compatibility of $\tilde{\alpha}$ with differentials reads
\begin{equation}\label{CobarCTwBijectionCD1}
d_A\circ \tilde{\alpha}\,=\,\tilde{\alpha}\circ (d_0+d_1+d_2)\,.
\end{equation}
Given $s^{-1}c\in \bar{C}[-1]$, \eqref{CobarCTwBijectionCD1} implies that 
\begin{eqnarray}
d_A\circ \alpha(c) &=& d_A\circ \tilde{\alpha}(s^{-1}c)\label{CobarCTwBijectionCD2}\\
&=& \tilde{\alpha}\circ (d_0+d_1+d_2)(s^{-1}c)\nonumber\\ 
&=& \tilde{\alpha}\big(h(c)-s^{-1}d_C(c)-(-1)^{|c'|}s^{-1}c'\otimes s^{-1}c''\big)\nonumber\\
&=& h(c)-\alpha\circ d_C(c) -(-1)^{|c'|}\alpha(c')\alpha(c'')\,,\nonumber
\end{eqnarray}
which is equivalent to $\alpha$ satisfies
$$\partial(\alpha)+\alpha\ast \alpha\,=\,u\circ h\,.$$
Thus, $\alpha\in \CTw(C,\,A)$.

Conversely, given $\alpha\in \CTw(C,\,A)$, we construct a $k$-linear map $\tilde{\alpha}(s^{-1}c)$ from $\bar{C}[-1]$ to $\bar{A}$ via $\tilde{\alpha}(s^{-1}c)=\alpha(c)$. By abuse of notation, we still denote the corresponding semi-augmented algebra morphism from $\cb(C)$ to $A$ by $\tilde{\alpha}$. To check the compatibility of $\tilde{\alpha}$ with differentials, it is enough to check on the generators $\bar{C}[-1]$ which is done in \eqref{CobarCTwBijectionCD2}. Therefore, $\tilde{\alpha}\in \Hom_{\mathtt{SemiDGA_{k/k}}}(\cb(C),\,A)$. Finally, it is clear that the two constructions above are inverse to each other.
\end{proof}

\subsection{Twisted tensor products}\label{subsecTwTensor}
Let $\alpha\in \bHom_{k}(C,\,A)$ be a morphism of degree $-1$. Let $d_{\alpha}^{r}$ denote the degree $-1$ morphism from $C\otimes A$ to itself given by
$$d_{\alpha}^{r}\defeq (\id_{C}\otimes \mu)\circ (\id_{C}\otimes \alpha \otimes \id_{A})\circ (\Delta\otimes \id_{A})\,.$$
Explicitly, for $a\in A$ and $c\in C$
\begin{equation}\label{TwTensorD1}
	d_{\alpha}^{r}(c\otimes a) = (-1)^{|c'|}c'\otimes \alpha (c'')\cdot a\,,
\end{equation}
where $\Delta(c)=c'\otimes c''$ in the Sweedler notation. Similarly, let $d_{\alpha}^{l}$ denote the degree $-1$ morphism from $A\otimes C$ to itself given by
$$d_{\alpha}^{l}\defeq (\mu\otimes \id_{C})\circ (\id_{A}\otimes \alpha \otimes \id_{C})\circ ( \id_{A}\otimes \Delta)\,.$$
Explicitly, for $a\in A$ and $c\in C$
\begin{equation}\label{TwTensorD2}
d_{\alpha}^{l}(a\otimes c) = (-1)^{|a|}a\cdot\alpha(c')\otimes c''\,,
\end{equation}

\begin{lemma}\label{TwTensorD3}
The morphism $d_{\alpha}^{r}$ satisfies
$$d_{C\otimes A}\circ d_{\alpha}^{r}+d_{\alpha}^{r}\circ d_{C\otimes A}\,=\,d_{\partial(\alpha)}^{r}
\quad and\quad
d_{\alpha}^{r}\circ d_{\alpha}^{r}=d_{\alpha\ast \alpha}^{r}\,.$$
Moreover, the morphism $d_{\alpha}^{l}$ satisfies
$$d_{A\otimes C}\circ d_{\alpha}^{l}+d_{\alpha}^{l}\circ d_{A\otimes C}\,=\,d_{\partial(\alpha)}^{l}
\quad and\quad
d_{\alpha}^{l}\circ d_{\alpha}^{l}=-d_{\alpha\ast \alpha}^{l}\,.$$
\end{lemma}

\begin{proof}
The statements about $d_{\alpha}^{r}$ are given in (the proofs of) \cite[Prop. 1.6.2 and Lemma 2.1.4]{LV}. The statements about $d_{\alpha}^{l}$ can be proved in a similar manner.
\end{proof}

Given a (uncurved) DG algebra $A$ and a (uncurved) DG coalgebra $C$, let $\alpha\in \bHom_{k}(C,\,A)$ be a twisting morphism. It turns out that there are two chain complexes:
$$
C\otimes_\alpha A\defeq (C\otimes A,\,d_{C\otimes A}+d_{\alpha}^{r})
\qquad\text{and}\qquad
A\otimes_\alpha C\defeq (A\otimes C,\,d_{A\otimes C}-d_{\alpha}^{l})\,,
$$
which are called the {\it right twisted tensor product} and the {\it left twisted tensor product}, respectively (see \cite[Sec. 2.1]{LV} for details). In general, $(d_{C\otimes A}+d_{\alpha}^{r})^2$ and $(d_{A\otimes C}-d_{\alpha}^{l})^2$ are nontrivial if either $A$ or $C$ is curved. Nevertheless, if one of $A$ and $C$ is curved and the other is uncurved, there exist the following two chain complexes that are related to the twisted tensor product constructions.

\begin{proposition}\label{TwTensorBimod}
Let $(A,\,d_A)$ be a DG algebra and $(C,\,d_C,\,h)$ be a curved DG coalgebra. Suppose $\alpha$ is a curved twisting morphism, i.e., $\partial(\alpha)+\alpha\ast \alpha=u\circ h$. Then the free $A$-bimodule generated by $C$ is a chain complex, denoted by $A\otimes_{\alpha}C\otimes_{\alpha}A$, with differential given by
\begin{equation}\label{TwTensorBimodD}
d_{A\otimes C\otimes A}+\id_{A}\otimes d_{\alpha}^{r}-d_{\alpha}^{l}\otimes\id_{A}\,.
\end{equation}
\end{proposition}

\begin{proof}
All we need to check is that \eqref{TwTensorBimodD} squares to zero:
\begin{eqnarray*}
&& (d_{A\otimes C\otimes A}+\id_{A}\otimes d_{\alpha}^{r}-d_{\alpha}^{l}\otimes\id_{A})\circ (d_{A\otimes C\otimes A}+\id_{A}\otimes d_{\alpha}^{r}-d_{\alpha}^{l}\otimes\id_{A})\\
&=& \id_A\otimes d_C^2\otimes \id_A+(d_A\otimes \id_{C\otimes A})\circ (\id_{A}\otimes d_{\alpha}^{r})+\id_A\otimes (d_{C\otimes A}\circ d_{\alpha}^{r})\\
&&-(d_{A\otimes C}\circ d_{\alpha}^{l})\circ \id_{A}-(\id_{A\otimes C}\otimes d_{A})\circ (d_{\alpha}^{l}\otimes\id_{A})\\
&& +(\id_{A}\otimes d_{\alpha}^{r})\circ (d_A\otimes \id_{C\otimes A})+\id_A\otimes (d_{\alpha}^{r}\circ d_{C\otimes A})+\id_A\otimes (d_{\alpha}^{r}\circ d_{\alpha}^{r})\\
&&-(\id_{A}\otimes d_{\alpha}^{r})\circ (d_{\alpha}^{l}\otimes\id_{A})-(d_{\alpha}^{l}\circ d_{A\otimes C})\otimes \id_{A}-(d_{\alpha}^{l}\otimes\id_{A})\circ (\id_{A\otimes C}\otimes d_{A})\\
&&-(d_{\alpha}^{l}\otimes\id_{A})\circ (\id_{A}\otimes d_{\alpha}^{r})+(d_{\alpha}^{l}\circ d_{\alpha}^{l})\otimes \id_{A}\\
&=&\id_{A}\otimes (d_{C\otimes A}\circ d_{\alpha}^{r}+d_{\alpha}^{r}\circ d_{C\otimes A}+d_{\alpha}^{r}\circ d_{\alpha}^{r})+d_A\otimes d_{\alpha}^{r}-d_A\otimes d_{\alpha}^{r}\\
&&-(d_{A\otimes C}\circ d_{\alpha}^{l}+d_{\alpha}^{l}\circ d_{A\otimes C}-d_{\alpha}^{l}\circ d_{\alpha}^{l})\otimes \id_{A}+d_{\alpha}^{l}\otimes d_{A}-d_{\alpha}^{l}\otimes d_{A}\\
&&+\id_A\otimes d_C^2\otimes \id_A-(\id_{A}\otimes d_{\alpha}^{r})\circ (d_{\alpha}^{l}\otimes\id_{A})-(d_{\alpha}^{l}\otimes\id_{A})\circ (\id_{A}\otimes d_{\alpha}^{r})
\end{eqnarray*}
By Lemma \ref{TwTensorD3}, 
$$d_{C\otimes A}\circ d_{\alpha}^{r}+d_{\alpha}^{r}\circ d_{C\otimes A}+d_{\alpha}^{r}\circ d_{\alpha}^{r}\,=\,d_{\partial{\alpha}+\alpha\ast\alpha}^{r}\,=\,d_{u\circ h}^{r}\,,$$
and
$$d_{A\otimes C}\circ d_{\alpha}^{l}+d_{\alpha}^{l}\circ d_{A\otimes C}-d_{\alpha}^{l}\circ d_{\alpha}^{l}\,=\,d_{\partial{\alpha}+\alpha\ast\alpha}^{l}\,=\,d_{u\circ h}^{l}\,.$$
By the definition, $\id_A\otimes d_C^2\otimes \id_A=d_{u\circ h}^{l}\otimes \id_{A}-\id_{A} \otimes d_{u\circ h}^{r}$. Moreover, one can check that
$$(\id_{A}\otimes d_{\alpha}^{r})\circ (d_{\alpha}^{l}\otimes\id_{A})\,=\,-(d_{\alpha}^{l}\otimes\id_{A})\circ (\id_{A}\otimes d_{\alpha}^{r})\,.$$
This finishes the proof of the deseried proposition.
\end{proof}

\begin{proposition}\label{CTwTensor}
Let $(A,\,\triangledown,\,\Theta)$ be a curved DG algebra and $(C,\,d)$ be a DG coalgebra. Suppose $\alpha$ is a curved twisting morphism, i.e., $\partial(\alpha)+\alpha\ast \alpha=-\Theta\epsilon$, then \eqref{TwTensorBimodD} induces a differential on $A\otimes_{A^e}(A\otimes C\otimes A)$. The resulting chain complex is denoted by $_{\alpha}A\otimes_{\alpha}C$.
\end{proposition}

\begin{proof}
Following the proof of Proposition \ref{TwTensorBimod}, one can show that, on $A\otimes C\otimes A$,
\begin{eqnarray*}
&& (d_{A\otimes C\otimes A}+\id_{A}\otimes d_{\alpha}^{r}-d_{\alpha}^{l}\otimes\id_{A})\circ (d_{A\otimes C\otimes A}+\id_{A}\otimes d_{\alpha}^{r}-d_{\alpha}^{l}\otimes\id_{A})\\
&=&	\triangledown^2\otimes \id_{C}\otimes \id_{A}+\id_{A}\otimes \id_{C}\otimes \triangledown^2-\id_{A}\otimes d_{\Theta\epsilon}^{r}+d_{\Theta \epsilon}^{l}\otimes \id_{A}\\
&=&(\Theta\cdot\mbox{--})\otimes \id_{C}\otimes \id_{A}-\id_{A}\otimes\id_{C}\otimes (\mbox{--}\cdot \Theta) \,.
\end{eqnarray*}
The desired proposition follows from the fact that for any $A$-bimodule $M$, $A\otimes_{A^e}M$ is isomorphic to the commutator quotient space of $M$.
\end{proof}

\section{QLC Koszul algebras}\label{secQLC}
In this section, we introduce the notion of an associative algebra presented by quadratic-linear-constant (QLC) relations and study its curved Koszul duality theory, which enables us to obtain semi-free resolutions of the QLC algebra.
\subsection{QLC algebras}
A {\it quadratic-linear-constant (QLC)} data is a non-negatively graded vector space $V$ together with a degree homogeneous subspace
$$R \subset k\oplus V\oplus V^{\otimes 2}\,.$$
The QLC algebra $A=A(V,\,R)\defeq T(V)/(R)$ is the quotient of the free associative algebra over $V$ by the two-sided ideal $(R)$ generated by $R$. We assume that $R$ satisfies the following two properties:
\begin{eqnarray}
&R\cap (k\oplus V)=\{0\}&\label{qlca1}\\
&(R)\cap (k\oplus V\oplus V^{\otimes 2})=R\,.&\label{qlca2}
\end{eqnarray}
The first condition amounts to the minimality of the space of generators of $A(V,\,R)$, and the second condition implies that $(R)$ does not create new QLC relations. Notice that a QLC algebra is semi-augmented, and the semi-augmentation map is induced by the augmentation map from $T(V)$ to $k$.

Let $q\,:\,T(V)\twoheadrightarrow  V^{\otimes 2}$ be the projection onto the quadratic part of tensor algebra, and $qR$ be the image of $R$ under $q$. \eqref{qlca1} guarantees that there exists a degree $0$ linear map $\psi\,:\,qR\rightarrow k\oplus V$ such that $R$ is the graph of $\psi$. Write $\psi=(-\phi,\,\theta)$, where $\phi\,:\,qR\rightarrow V$ and $\theta\,:\,qR\rightarrow k$. One has
\begin{equation}\label{QLCgraph}
R\,=\,\{x-\phi(x)+\theta(x)\,\mid\,x\in qR\}\,.
\end{equation}

Since $(V,\,qR)$ is a quadratic data (in the sense of \cite[Sec. 3.1]{LV}), We denote by $qA$ the associated quadratic algebra, i.e., $qA=T(V)/(qR)$. Notice that $qA$ is weight graded:
\begin{equation}\label{KAWeightDecomp}
qA\,=\,\bigoplus_{n\in \mathbb{N}}qA^{(n)}\,=\,k\oplus V \oplus (V^{\otimes 2}/qR) \oplus \cdots \oplus \big(V^{\otimes n}/\sum_{i+2+j=n}V^{\otimes i}\otimes qR \otimes V^{\otimes j}\big)\oplus \cdots\,.
\end{equation}

\subsection{Koszul dual curved coalgebra}\label{subsecKDCCA}
Recall from \cite[Sec. 3.2]{LV}, the {\it Koszul dual coalgebra} $(qA)^\dual=C(sV,\,s^2qR)$ of the quadratic algebra $qA$ is a non-negtaively graded, connected coalgebra equipped with a weight grading given by
\begin{equation}\label{KDCWeightDecomp}
(qA)^{\dual}\,=\,\bigoplus_{n\in \mathbb{N}}(qA)^{\dual,\,(n)}\,=\,k\oplus sV \oplus s^2qR \oplus \cdots \oplus \big(\bigcap_{i+2+j=n}(sV)^{\otimes i}\otimes s^2qR \otimes (sV)^{\otimes j}\big)\oplus \cdots\,.
\end{equation}
$\phi$ induces a degree $-1$ map $\tilde{\phi}$ from $(qA)^{\dual}$ to $sV$ via the composition:
\begin{equation}\label{phiqadual}
\tilde{\phi}\,:\,(qA)^{\dual}=C(sV,\,s^2qR)\twoheadrightarrow s^2qR\xrightarrow{s^{-1}\phi} sV\,,
\end{equation}
and $\theta$ induces a degree $-2$ map $\tilde{\theta}$ from $(qA)^{\dual}$ to $k$ via the composition:
\begin{equation}\label{thetaqadual}
\tilde{\theta}\,:\,(qA)^{\dual}=C(sV,\,s^2qR)\twoheadrightarrow s^2qR\xrightarrow{s^{-2}\theta} k\,.
\end{equation}

\begin{proposition}\label{KoszulCurCAg}
$((qA)^\dual,\,\tilde{\phi},\,\tilde{\theta})$ is a curved DG coalgebra.
\end{proposition}

We call this curved DG coalgebra the {\it Koszul dual curved coalgebra} of $A$. 

\begin{lemma}\cite[Ch. 5 Prop. 1.1]{PP05}\label{PP05LemmaCC} The map $\psi=(-\phi,\,\theta)$ satisfies the following conditions
\begin{eqnarray}
&(\phi\otimes \id-\id\otimes \phi)(V\otimes qR\cap qR\otimes V)\,\subset\,qR& \label{PP05LemmaCC1}\\
&\phi\circ(\phi\otimes \id-\id\otimes \phi)(V\otimes qR\cap qR\otimes V)\,=\,(\theta\otimes \id-\id\otimes \theta)(V\otimes qR\cap qR\otimes V)&\label{PP05LemmaCC2}\\
&\theta\circ(\phi\otimes \id-\id\otimes \phi)(V\otimes qR\cap qR\otimes V)\,=\,0\,.&\label{PP05LemmaCC3}
\end{eqnarray}
\end{lemma}

\begin{lemma}\label{curvatureCoder}
Let $h$ be a linear map of degree $-2$ from a graded coalgebra $C$ to $k$, then $(h\otimes \id- \id\otimes h)\circ\Delta$ is a coderivation on $C$.
\end{lemma}

\begin{proof}
\begin{eqnarray*}
&&\big((h\otimes \id- \id\otimes h)\Delta\otimes\id+\id\otimes (h\otimes \id- \id\otimes h)\Delta\big)\circ\Delta\\
&=&(h\otimes \id\otimes \id-\id\otimes h \otimes \id + \id\otimes h \otimes \id-\id\otimes \id \otimes h)\circ \Delta^{(3)}\\
&=&(h\otimes \id\otimes \id-\id\otimes \id \otimes h)\circ \Delta^{(3)}\\
&=&(h\otimes \Delta-\Delta\otimes h)\circ \Delta\\
&=&\Delta\circ\big( (h\otimes \id- \id\otimes h)\circ\Delta\big)
\end{eqnarray*}
\end{proof}

\begin{proof}[Proof of Proposition \ref{KoszulCurCAg}]
By the construction, there exist a unique coderivation on $T^c(sV)$ extending $\tilde{\phi}$ (which is denoted by $\tilde{\phi}$ as well). First of all, we need to show that $\tilde{\phi}((qA)^\dual)\subset (qA)^\dual$. Recall from \eqref{KDCWeightDecomp} that there is a weight decomposition on $(qA)^\dual$. By the formula for coderivations on $T^c(sV)$, one can check that \eqref{PP05LemmaCC1} in Lemma \ref{PP05LemmaCC} implies $\tilde{\phi}((qA)^{\dual,\,(3)})\subset (qA)^{\dual,\,(2)}$. Moreover, $\tilde{\phi}((qA)^{\dual,\,(2)})\subset (qA)^{\dual,\,(1)}$ by the construction of $\tilde{\phi}$. We shall show that $\tilde{\phi}((qA)^{\dual,\,(n+1)})\subset (qA)^{\dual,\,(n)}$ by induction. 

Let $p_1$ (resp., $p_{n-1}$) denote the canonical projection from $T^c(sV)$ onto $sV$ (resp., $(sV)^{\otimes n-1}$), and one has 
$$(p_{n-1}\otimes p_1)\circ\Delta \mid_{(sV)^{\otimes n}}\,=\,\id\mid_{(sV)^{\otimes n}}\,=\,(p_1\otimes p_{n-1})\circ\Delta \mid_{(sV)^{\otimes n}}\,.$$
It is clear that $\tilde{\phi}((qA)^{\dual,\,(n+1)})\subset (sV)^{\otimes n}$. Hence, on $(qA)^{\dual,\,(n+1)}$,
\begin{equation}\label{DiffCurCAg1}
\tilde{\phi}\,=\,(p_{n-1}\otimes p_1)\circ\Delta\circ \tilde{\phi}\,=\,\big((p_{n-1}\circ \tilde{\phi})\otimes p_1+p_{n-1}\otimes (p_1\circ \tilde{\phi})\big)\circ\Delta\,,
\end{equation}
and
\begin{equation}\label{DiffCurCAg2}
\tilde{\phi}\,=\,(p_1\otimes p_{n-1})\circ\Delta\circ \tilde{\phi}\,=\,\big((p_1\circ \tilde{\phi})\otimes p_{n-1}+p_1\otimes (p_{n-1}\circ \tilde{\phi})\big)\circ\Delta\,.
\end{equation}
The weight grading on $(qA)^\dual$ implies that 
\begin{eqnarray*}
& (\id\otimes p_1)\circ \Delta((qA)^{\dual,\,(n+1)})\subset (qA)^{\dual,\,(n)}\otimes (qA)^{\dual,\,(1)}\,,&\\
\text{and}& \quad(p_{n-1}\otimes \id)\circ \Delta((qA)^{\dual,\,(n+1)})\subset (qA)^{\dual,\,(n-1)}\otimes (qA)^{\dual,\,(2)}\,.&
\end{eqnarray*}
 By \eqref{DiffCurCAg1} and the inductive hypothesis $\tilde{\phi}((qA)^{\dual,\,(n)})\subset (qA)^{\dual,\,(n-1)}$,  we obtain that $\tilde{\phi}((qA)^{\dual,\,(n+1)})\subset (qA)^{\dual,\,(n-1)}\otimes sV$ . Similarly, one could deduce $\tilde{\phi}((qA)^{\dual,\,(n+1)})\subset sV\otimes  (qA)^{\dual,\,(n-1)}$ from \eqref{DiffCurCAg2}. Therefore,
 \begin{eqnarray*}
 \tilde{\phi}((qA)^{\dual,\,(n+1)})& \subset & (qA)^{\dual,\,(n-1)}\otimes sV\bigcap sV\otimes  (qA)^{\dual,\,(n-1)}\\
 & \subset & \big(\bigcap_{i+2+j=n-1}(sV)^{\otimes i}\otimes s^2qR \otimes (sV)^{\otimes j+1}\big)\bigcap \big(\bigcap_{i+2+j=n-1}(sV)^{\otimes i+1}\otimes s^2qR \otimes (sV)^{\otimes j}\big)\\
 &=& \bigcap_{i+2+j=n}(sV)^{\otimes i}\otimes s^2qR \otimes (sV)^{\otimes j}\\
 &=&  (qA)^{\dual,\,(n)}\,.
 \end{eqnarray*}
Therefore, $\tilde{\phi}$ is a coderivation on $(qA)^\dual$.

Next, by Lemma \ref*{curvatureCoder}, both $(\tilde{\theta}\otimes \id- \id\otimes \tilde{\theta})\circ\Delta$ and  $\tilde{\phi}^2=\frac{1}{2}[\tilde{\phi},\,\tilde{\phi}]$ are coderivations on $(qA)^\dual$. Thus, $\tilde{\phi}^2-(\tilde{\theta}\otimes \id- \id\otimes \tilde{\theta})\circ\Delta$ is completely determined by its corestriction to $sV$, which vanishes by \eqref{PP05LemmaCC2}. Thus, 
$$\tilde{\phi}^2\,=\,(\tilde{\theta}\otimes \id- \id\otimes \tilde{\theta})\circ\Delta\,.$$

Lastly, it is clear that \eqref{PP05LemmaCC3} implies $\tilde{\theta}\circ\tilde{\phi}=0$. This finishes the proof of the desired proposition.
\end{proof}

\begin{remark}
In \cite{Pos93}, the author shows that $(qA)^!$, the graded linear dual of $(qA)^\dual$, is a curved DG algebra under some finiteness assumptions.
\end{remark}

\subsection{The cobar construction of the Koszul dual curved coalgebra}\label{subsecCobarKDCCA}
Consider the following composite map from the Koszul dual curved coalgebra  $(qA)^{\dual}$ to $A$, 
$$\kappa\,:\,(qA)^{\dual}=C(sV,\,s^2qR)\twoheadrightarrow sV \xrightarrow{s^{-1}}V\hookrightarrow A=A(V,\,R)\,.$$

\begin{lemma}\label{KoszulCTW}
The map $\kappa$ is a curved twisting morphism in $\bHom_{k}((qA)^{\dual},\,A)$.
\end{lemma}

\begin{proof}
Notice that the map $\kappa\circ \tilde{\phi}+\kappa\ast \kappa$ is zero everywhere except on $s^2qR$, and its image is precisely $\{-x+\phi(x)\,\mid\,x\in qR\}$, which equals $\{\theta(x)\,\mid\,x\in qR\}$ in $A$.
Therefore, one has 
\begin{equation}\label{CMCkappa}
\kappa\circ \tilde{\phi}+\kappa\ast \kappa\,=\,u\circ \tilde{\theta}\,,
\end{equation}
which implies that $\kappa$ is a curved twisting morphism.
\end{proof}

By Proposition \ref{CobarCTwBijection}, $\kappa$ induces a morphism of semi-augmented algebras $g_\kappa\,:\,\cb(C)\rightarrow A$, where $C=((qA)^\dual,\,\tilde{\phi},\,\tilde{\theta})$ is the  Koszul dual curved coalgebra. Recall from \cite[Thm. 3.4.4]{LV}, the quadratic algebra $qA$ is called {\it Koszul} if the canonical projection from $(\cb((qA)^{\dual}),\,d_2)$ to $qA$, induced by
$$(qA)^{\dual}=C(sV,\,s^2qR)\twoheadrightarrow sV \xrightarrow{s^{-1}}V\hookrightarrow qA=A(V,\,qR)\,,$$
is a quasi-isomorphism, where $d_2$ is the differential induced by the reduced comultiplication on $\bar{C}$. The following theorem is a generalization of \cite[Thm. 3.6.3]{LV} to the case of QLC algebras.
\begin{theorem}\label{KoszulFreeRes}
Let $A$ be a QLC algebra, and $C=((qA)^\dual,\,\tilde{\phi},\,\tilde{\theta})$ be its Koszul dual curved coalgebra. Suppose $qA$ is Koszul, then the morphism of DG algebras $g_\kappa\,:\,\cb(C)\rightarrow A$ is a quasi-isomorphism.
\end{theorem}

Notice that $A$ is equipped with a weight filtration $F_r'$  induced by the weight filtration on $T(V)$, i.e.,  $F_r'$ is given by the image of $T^{\leqslant r}(V)$ under the canonical projection $T(V) \twoheadrightarrow A$. Let $\mathrm{gr}A$ denote the graded algebra associated to this filtration. The following proposition is based on Lemma \ref{PP05LemmaCC} and equivalent characterizations of the Koszulness \cite[Ch. 2 Thm. 4.1]{PP05}.
\begin{proposition}\cite[Ch. 5 Thm 2.1]{PP05}\label{PBWThm}
Let $A=A(V,\,R)$ be a QLC algebra, and $qA$ be the associated quadratic algebra. Suppose $qA$ is Koszul, then $\mathrm{gr}A$ is isomorphic to $qA$.
\end{proposition}

\begin{remark}
Proposition \ref{PBWThm} is called the nonhomogeneous Poincar\'{e}-Birkhoff-Witt (PBW) theorem in \cite{PP05}. For example, let $A=\U(\mfa)$, the universal enveloping algebra of a finite dimensional Lie algebra $\mfa$. Then $\Sym(\mfa) = qA \cong \mathrm{gr}A$ is nothing but the classical Poincar\'{e}-Birkhoff-Witt theorem for Lie algebra $\mfa$.
\end{remark}

\begin{proof}[Proof of Theorem \ref{KoszulFreeRes}]
Since $C$ is weight graded, the graded algebra $T(\bar{C}[-1])$ is weight graded as well. We consider the filtration $F_r$ of $\cb(C)$ defined by its weights: the elements of $F_r$ are the elements of weight less than or equal to $r$. It is clear that $g_\kappa$ is morphism of filtered complexes. 

Let $E_{rs}^\bullet$ (resp., $E_{rs}^{' \bullet}$) be the spectral sequence associated to the weight filtration on $\cb(C)$ (resp., weight filtration on $A$). Since both filtrations are bounded below and exhaustive, the classical convergence theorem of spectral sequences \cite[Thm. 5.5.1]{Weibel94} implies that $E_{rs}^\bullet$ and $E_{rs}^{' \bullet}$ converge to $\H_\bullet(\cb(C))$ and $\H_\bullet(A)=A$, respectively, and the map induced by $g_\kappa$ on homologies 
$$\H_\bullet(g_\kappa)\,:\,\H_\bullet(\cb(C))\,\rightarrow\,A$$
is compatible with the corresponding map between $\{E_{rs}^\bullet\}$ and $\{E_{rs}^{' \bullet}\}$.

$E_{rs}^{0}$ is made up of the elements in $\cb(C)$ of degree equal to $r+s$ and weight equal to $r$. The three components of the differential map $d_{\cb(C)}=d_0+d_1+d_2$ satisfy
$$d_2\,:\,F_r\rightarrow F_r \qquad d_1\,:\,F_r\rightarrow F_{r-1} \qquad d_0\,:\,F_r\rightarrow F_{r-2}\,.$$
Thus, the differential $d^0$ on $E_{rs}^{0}$ is given by $d_2$. By proposition \ref{PBWThm}, $E_{rs}^{' 0}$ is made up of the elements in $qA$ of degree equal to $r+s$ and weight equal to $r$. It is easy to see that the corresponding map from $E^0_{rs}$ to $E_{rs}^{' 0}$ is precisely the canonical projection from $\cb((qA)^{\dual})$ to $qA$, which is a quasi-isomorphism for $qA$ is Koszul. Thus, the corresponding map from $E^1_{rs}$ to $E_{rs}^{' 1}$ is an isomorphism. The desired theorem follows from the comparison theorem of spectral sequences \cite[Thm. 5.2.12]{Weibel94}.
\end{proof}

\begin{remark}
Polishchuk and Positselski have a similar result in \cite{PP05}. However, their definition of cobar constructions is slight different than what we have. Those two definitions are equivalent if we assume that $V$ is a finite dimensional vector space. 
\end{remark}

Showing the condition \eqref{qlca2} could be difficult. However, if $qA$ is Koszul, it suffices to check a weaker condition:

In the proof of Lemma \ref{PP05LemmaCC}, the authors only use the following condition
\begin{equation}\label{qlca2W}
(V\otimes R\oplus R\otimes V)\cap (k\oplus V\oplus V^{\otimes 2})\,\subset\,R\,,
\end{equation}
which is weaker than \eqref{qlca2}. Moreover, if $qA$ is Koszul, we still have the PBW isomorphism
$$\mathrm{gr}A\,\cong\,qA\,.$$
It follows from the weight $0$ and $1$ components of this isomorphism that 
\begin{equation}\label{qlca1S}
(R)\cap (k\oplus V)=\{0\}\,.
\end{equation}
Additionally, in weight $2$, the PBW isomorphism implies that
\begin{equation}\label{qlca2W2}
qR\,=\,q\big((R)\cap (k\oplus V\oplus V^{\otimes 2})\big)\,.
\end{equation}
These two conditions imply \eqref{qlca2}. Indeed, if $R\subsetneq \big((R)\cap (k\oplus V\oplus V^{\otimes 2})\big)$, by \eqref{qlca2W2}, one can find $x,\,y \in (R)\cap (k\oplus V\oplus V^{\otimes 2})$ such that $x\neq y$ but $q(x)=q(y)$. Then $0\neq x-y \in (R)\cap (k\oplus V)$, which contradicts the condition \eqref{qlca1S}.

\subsection{Koszul resolution}
Let $A$ be a QLC algebra satisfying the conditions \eqref{qlca1} and \eqref{qlca2} (or equivalently, \eqref{qlca2W}). By Proposition \ref{KoszulCurCAg} and Lemma \ref{KoszulCTW}, the Koszul dual coalgebra $(qA)^\dual$ of $A$ is a curved DG coalgebra and there is a curved twisting morphism $\kappa\,:\,(qA)^\dual\rightarrow A$. In addition, by Proposition \ref{TwTensorBimod}, there is a chain complex of  $A$-bimodules $A\otimes_\kappa (qA)^\dual\otimes_\kappa A$, called the {\it total Koszul complex} of the QLC algebra $A$. The following theorem generalizes the Koszul resolution of quadratic algebras to QLC algebras.

\begin{theorem}\label{KoszulRes}
Let $A$ be a QLC algebra, and $(qA)^\dual$ be its Koszul dual curved coalgebra. Suppose $qA$ is Koszul, then the total Koszul complex $A\otimes_\kappa(qA)^\dual\otimes_\kappa A$ is a resolution of $A$ in the category of $A$-bimodules. 
\end{theorem}

\begin{proof}
Consider the following composition
$$\xi\,:\,A\otimes_\kappa (qA)^\dual\otimes_\kappa A\,\xrightarrow{\id\otimes \epsilon\otimes\id}\,A\otimes k\otimes A\,\cong\,A\otimes A\,\xrightarrow{\mu}\,A\,.$$
It is straightforward to check that $\xi$ is a morphism of $A$-bimodules. We are going to show that this morphism is a quasi-isomorphism.

The weight filtrations on $A$ and $(qA)^\dual$ induce a filtration on $A\otimes_\kappa\otimes (qA)^\dual\otimes_\kappa A$:
$$F_r\big(A\otimes_\kappa(qA)^\dual\otimes_\kappa A\big)\,=\,\sum_{r_1+r_2+r_3=r}F_{r_1}A\otimes F_{r_2}(qA)^\dual\otimes F_{r_3}A\,.$$
Let $E_{rs}^\bullet$ (resp., $E_{rs}^{' \bullet}$) be the spectral sequence associated to the filtration on $A\otimes_\kappa\otimes (qA)^\dual\otimes_\kappa A$ (resp., weight filtration on $A$). Since both filtrations are bounded below and exhaustive, and $\xi$ is a morphism of filtered complexes, by the classical convergence theorem of spectral sequences, the induced map on homologies 
$$\H_\bullet(\xi)\,:\,\H_\bullet[A\otimes_\kappa\otimes (qA)^\dual\otimes_\kappa A]\,\rightarrow\,\H_\bullet(A)\,=\,A$$
is compatible with the corresponding map between $\{E_{rs}^\bullet\}$ and $\{E_{rs}^{' \bullet}\}$.

Since 
$$\mathrm{gr}F_\bullet\big(A\otimes (qA)^\dual\otimes A\big)\,\cong\,\mathrm{gr}A\otimes \mathrm{gr}(qA)^\dual\otimes \mathrm{gr}A\,\cong\,qA\otimes (qA)^\dual\otimes qA\,,$$
$E^0_{rs}$ is made up the elements of $qA\otimes (qA)^\dual\otimes qA$ with degree euqal to $r+s$ and total weight equal to $r$. Recall the differential on $A\otimes_\kappa(qA)^\dual\otimes_\kappa A$ is given by
$$d_{C}+\id_{A}\otimes d_{\kappa}^{r}-d_{\kappa}^{l}\otimes\id_{A}\,,$$
where $d_C$ is induced by $\tilde{\phi}$ which decreases the filtration by $1$. Since $\kappa$ preserves the weight filtrations, the differential on the $E^0_{rs}$ is given by $\id_{qA}\otimes d_{\alpha}^{r}-d_{\alpha}^{l}\otimes\id_{qA}$, where $\alpha$ is the canonical twisting morphism from $(qA)^\dual$ to $qA$. Moreover, the corresponding map from $E^0_{rs}$ to $E_{rs}^{' 0}$ is given by the following composition
$$qA\otimes_\alpha (qA)^\dual\otimes_\alpha qA\,\xrightarrow{\id\otimes \epsilon\otimes\id}\,qA\otimes k\otimes qA\,\cong\,qA\otimes qA\,\xrightarrow{\mu}\,qA\,.$$
Since $qA$ is Koszul, the map above is a quasi-isomorphism by \cite[Ex. 2.7.6]{LV}. Thus the corresponding map from $E^1_{rs}$ to $E_{rs}^{' 1}$ is an isomorphism. It follows from the comparison theorem of spectral sequences that $\H_\bullet(\xi)$ is an isomorphism. 
\end{proof}

\begin{example}[The Weyl algebra]
Let $V$ be the two dimensional vector space generated by $\{x,\,y\}$, and $R$ is the one dimensional space generated by $\{yx-xy-1\}$. It is clear that both conditions \eqref{qlca1} and \eqref{qlca2W} are satisfied. $A=T(V)/(R)$ is the first Weyl algebra. Since $qR$ is the one dimensional space generated by $\{yx-xy\}$, $qA$ is the symmetric algebra $\Sym(V)$ which is Koszul. The Koszul dual coalgebra is $(qA)^\dual=\bSym^c(sV)$, the cofree cocommutative coalgebra over $sV$. As  vector spaces, $\bSym^c(sV)\cong \Lambda^\bullet (V)$. Thus the free resolution provided by Theorem \ref{KoszulRes} can be written as
\begin{equation}\label{KoszulResWeyl1}
A\otimes \Lambda^2 V\otimes A\,\xrightarrow{\partial_2}\,A\otimes V\otimes A\,\xrightarrow{\partial_1}\,A\otimes A\,,
\end{equation}
where the differentials are given by
$$\partial_2(1\otimes x\wedge y\otimes 1)\,=\,-1\otimes x\otimes y +1\otimes y\otimes x-x\otimes y\otimes 1+y\otimes x\otimes 1\,,$$
$$\partial_1(1\otimes x\otimes 1)\,=\,1\otimes x-x\otimes 1\qquad\text{and}\qquad \partial_1(1\otimes y\otimes 1)\,=\,1\otimes y-y\otimes 1\,.$$
This is the length-two resolution constructed in \cite[Sec. 1]{Kassel92}, which can be used to compute the Hochschild homology of the Weyl algebra $A$ (see \cite[Sec. 5.10]{Kassel06}). 

The Weyl algebra is a special case of the next example.
\end{example}

\begin{example}[The universal enveloping algebra of a ``unital'' Lie algebra]
Let 
$$0\rightarrow k\cdot a \rightarrow \tilde{\g} \rightarrow \g \rightarrow 0$$ 
be a central extension of a finite dimensional Lie algebra $\g$ by a trivial $\g$-module $k\cdot a \cong k$. It is known that the equivalence classes of central extensions of $\g$ by the trivial $\g$-module $k$ is in one-to-one correspondence with the second Lie cohomology group of $\g$ with trivial coefficients (see, for example, \cite[Ch. 7]{Weibel94}). Let $\U\tilde{\g}$ be the universal enveloping algebra of $\tilde{\g}$ and consider the quotient $A=\U\tilde{\g}/(a-1)$. It can be shown that $A$ is a QLC algebra with $\g$ as the space of generators and $R$ is generated by
\begin{equation}\label{UUEArelations}
x\otimes y -y\otimes x-[x,\,y]-h(x,\,y)\,,
\end{equation}
where $x,\,y\in \g$ and $h(x,\,y)$ is a $2$-cocycle describing the extension.

Let $\mathit{Ass}$ be the binary quadratic operad governing nonunital associative algebras, and $\mathit{Lie}$ be the binary quadratic operad governing Lie algebras. Let $\mu\in \mathit{Ass}(2)$ be the generator of $\mathit{Ass}$ and $\nu\in \mathit{Lie}(2)$ be the generator of $\mathit{Lie}$. Recall that there is a morphism of operads from $\mathit{Lie}$ to $\mathit{Ass}$ given by $\nu\mapsto \mu-\mu^{(12)}$, which induces the Lie algebra functor ${\mathcal Lie}$ from the category of associative algebras to the category of Lie algebras. 

The notion of a unital version of a binary quadratic operad was introduced by Hirsh and Mill\`{e}s in \cite{HM12}. Let $\mathit{uAss}$ be a operad with generators $u\in \mathit{uAss}(0)$ and $\mu\in \mathit{uAss}(2)$ and relations $\mu\circ_1 \mu - \mu\circ_2 \mu$, $\mu\circ_1 u-\id$, and $\mu\circ_2 u-\id$. Then $\mathit{uAss}$, which is a unital version of $\mathit{Ass}$, governs unital associative algebras. Similarly, let $\mathit{cLie}$ be a operad with generators $c\in \mathit{cLie}(0)$ and $\nu\in \mathit{cLie}(2)$ satisfying $\nu^{(12)}=-\nu$ and relations $\nu\circ_1 \nu+(\nu\circ_1 \nu)^{(123)}+(\nu\circ_1 \nu)^{(132)}$ and $\nu\circ_1 c$. Then $\mathit{cLie}$, which is a unital version of $\mathit{Lie}$, governs Lie algebras with a designated central element. $\tilde{\g}$ is an example of a $\mathit{cLie}$-algebra with $a$ being the designated central element. By abuse of terminology we call $\tilde{\g}$ a {\it unital} Lie algebra and view $a\in \tilde{\g}$ as the ``unit''. The morphism of operads from $\mathit{cLie}$ to $\mathit{uAss}$ given by $\nu\mapsto \mu-\mu^{(12)}$ and $c\mapsto u$ induces the unital Lie algebra functor ${u\mathcal Lie}$ from the category of unital associative algebras to the category of Lie algebras with a designated central element such that the unit is sent to the designated central element. One can mimic the construction of universal enveloping algebras and show that ${u\mathcal Lie}$ has a left adjoint functor, and the image of $\tilde{\g}$ under this functor is precisely $A$. Therefore, we can view $A$ as the universal enveloping algebra of the ``unital'' Lie algebra $\tilde{\g}$. We refer to \cite{Sridharan61} for an interpretation of $A$ in Lie algebra representations.

The Koszul dual coalgebra $(qA)^\dual$ of $A$ is the Chevalley-Eilenberg DG coalgebra of $\g$, and the corresponding curved twisting morphism $\kappa\,:\,(qA)^\dual\rightarrow A$ is induced by $\g[1] \rightarrow \g$. The relation \eqref{UUEArelations} implies that $A$ is a right $\g$-module with the action given by 
$$m\cdot g\,=\,mg-gm\,,$$
for $m\in A$ and $g\in \g$.
Using the resolution $A\otimes_\kappa (qA)^\dual \otimes_\kappa A$ provided by Theorem \ref{KoszulRes}, one can identify the Hochschild (co)homology of $A$ with the Lie algebra (co)homology of $\g$ with coefficients in $A$ ({\it cf.} \cite[Prop. 5.3]{Sridharan61} and \cite[Cor. 5]{Kassel88}), that is, 
$$\HH_\bullet(A)\,\cong\,\H_\bullet(\g,\,A)
\qquad\text{and}\qquad 
\HH^\bullet(A,\,A)\,\cong\,\H^\bullet(\g,\,A)\,.$$
\end{example}

\section{Cyclic (co)homology of QLC algebras}\label{secCycHo}
We give the construction of (reduced) cyclic bicomplexes of an augmented curved DG algebra. To the best of our knowledge, this material has not appeared in the literature. The main result in this section is a generalization of a theorem due to Feigin and Tsygan.
\subsection{Hochschild complex}\label{subsecHoch}
Recall that the bar construction $\bB(A)$ of an augmented curved DG algebra $(A,\,\triangledown,\,\Theta)$ is a DG coalgebra. Moreover, the following composition
$$\pi\,:\, \bB(A)\,\twoheadrightarrow\,\bar{A}[1]\,\xrightarrow{s^{-1}}\,\bar{A}\,\hookrightarrow\,A\,,$$
is a curved twisting morphism. By Proposition \ref{CTwTensor}, $_{\pi}A\otimes_{\pi}\bB(A)$ is a chain complex, called the {\it Hochschild complex} of $A$, and the differential is given by the sum of the following three terms:
\begin{equation}\label{HochD0}
d^{H}_{0}(a_{0},\,sa_{1},\,\cdots,\,sa_{n})\,=\,\sum_{i=0}^{ n}(-1)^{|a_{0}|+\cdots+|a_{i}|+i+1}(a_{0},\,sa_{1},\,\cdots,\,sa_{i},\,s\Theta,\,sa_{i+1},\,\cdots,\,sa_{n})\,,
\end{equation}
\begin{eqnarray}\label{HochD1}
d^{H}_{1}(a_{0},\,sa_{1},\,\cdots,\,sa_{n}) &=& (\triangledown(a_{0}),\,sa_{1},\,\cdots,\,sa_{n})\\
&+&\sum_{i=1}^{n}(-1)^{|a_{0}|+\cdots+|a_{i-1}|+i}(a_{0},\,sa_{1},\,\cdots,\,s\triangledown(a_{i}),\,\cdots,\,sa_{n})\,,\nonumber
\end{eqnarray}
and
\begin{eqnarray}\label{HochD2}
d^{H}_{2}(a_{0},\,sa_{1},\,\cdots,\,sa_{n}) &=&(-1)^{|a_{0}|+1}(a_{0}a_{1},\,sa_{2},\,\cdots,\,sa_{n}) \\
&+&\sum_{i=1}^{n-1}(-1)^{|a_{0}|+\cdots+|a_{i}|+i-1}(a_{0},\,sa_{1},\,\cdots,\,s(a_{i}a_{i+1}),\,\cdots,\,sa_{n})\nonumber\\
&+&(-1)^{(|a_{n}|+1)(|a_{0}|+\cdots+|a_{n-1}|+n-1)}(a_{n}a_{0},\,sa_{1},\,\cdots,\,sa_{n-1})\,.\nonumber
\end{eqnarray}
The associated homology is called the {\it Hochschild homology} of the curved DG algebra $A$ and is denoted by $\HH_{\bullet}(A)$. Moreover, the assocaited homology of the subcomplex $_{\pi}\bar{A}\otimes_{\pi}\bB(A)$ is called the {\it reduced Hochschild homology} of the curved DG algebra $A$ and is denoted by $\rHH_{\bullet}(A)$.

\begin{remark}
Given a DG coalgebra $C$, one can construct the DG bicomodule of differentials $\Omega^1(C)$ over $C$ (see \cite[Sec. 4]{Qui89} for details). It can be shown that, up to a shift in degree, $_{\pi}\bar{A}\otimes_{\pi}\bB(A)$  with the differential described above is isomorphic to $\Omega^1(\bB(A))^\n$, where $(\mbox{--})^\n$ denotes the cocommutator subspace of a bicomodule.
\end{remark}

\subsection{Cyclic homology}\label{subsecHC}
Let $\rbB(A)$ be the subspace of $\bB(A)$ whose underlying graded vector space is $\overline{T^{c}}(\bar{A}[1])$, and it is immediate to check that $\rbB(A)$ is a subcomplex of $\bB(A)$. Let $T_n$ denote the cyclic group action on $\bar{A}[1]^{\otimes n}$, i.e.,
\begin{equation}\label{TDef1}
T_n(sa_{1},\,\cdots,\,sa_{n})\,=\,(-1)^{(|a_{n}|+1)(|a_{1}|+\cdots+|a_{n-1}|+n-1)}(sa_{n},\,sa_{1},\,\cdots,\,sa_{n-1})\,.
\end{equation}
Consider the map $1-T$ from $\rbB(A)$ to $_{\pi}\bar{A}[1]\otimes_{\pi}\bB(A)$ whose restriction to $\bar{A}[1]^{\otimes n}$ is
\begin{eqnarray}\label{TDef2}
&&(1-T_n)(sa_{1},\,\cdots,\,sa_{n})\\
&=&(sa_{1},\,\cdots,\,sa_{n})-(-1)^{(|a_{n}|+1)(|a_{1}|+\cdots+|a_{n-1}|+n-1)}(sa_{n},\,sa_{1},\,\cdots,\,sa_{n-1})\,,\nonumber
\end{eqnarray}
and the map $N$ from $_{\pi}\bar{A}[1]\otimes_{\pi}\bB(A)$ to $\rbB(A)$ whose restriction to $\bar{A}[1]^{\otimes n}$ is
\begin{eqnarray}\label{NDef}
&&(1+T_n+\dots+T_n^{n-1})(sa_{1},\,\cdots,\,sa_{n})\\
&=&\sum_{i=1}^{n}(-1)^{(|a_{i}|+\cdots+|a_{n}|+n-i+1)(|a_{1}|+\cdots+|a_{i-1}|+i-1)}(sa_{i},\,\cdots,\,sa_{n},\,sa_{1},\,\cdots,\,sa_{i-1})\,,\nonumber
\end{eqnarray}

\begin{lemma}\label{CycBicomplexLemma}
$1-T$ and $N$ are morphisms between complexes.
\end{lemma}

\begin{proof}
The isomorphism between $_{\pi}\bar{A}[1]\otimes_{\pi}\bB(A)$ and $_{\pi}\bar{A}\otimes_{\pi}\bB(A)$ is given by $s\,:\,a\otimes b\mapsto sa\otimes b$, hence the differential on $_{\pi}\bar{A}[1]\otimes_{\pi}\bB(A)$ is $-d_0^H-d_1^H-d_2^H$. We will show that
\begin{eqnarray}\label{CycBicomplexLemma3}
(-d_0^H-d_1^H-d_2^H)\circ(1-T) &=& (1-T)\circ(d_0+d_1+d_2)\\
\text{and}\qquad\qquad(d_0+d_1+d_2)\circ N&=& N\circ (-d_0^H-d_1^H-d_2^H)\,.\nonumber
\end{eqnarray}

Let
$$d_{2,\,i}(sa_{1},\,\cdots,\,sa_{n})\,=\,
\begin{cases}
(-1)^{|a_{1}|+\cdots+|a_{i}|+i-1}(sa_{1},\,\cdots,\,s(a_{i}a_{i+1}),\,\cdots,\,sa_{n}) & 1\leqslant i\leqslant n-1\\
(-1)^{(|a_{n}|+1)(|a_{1}|+\cdots+|a_{n-1}|+n-1)+|a_{n}|}(s(a_{n}a_{1}),\,sa_{2},\,\cdots,\,sa_{n-1}) & i=n\,.
\end{cases}
$$
Then $d_2=d_{2,\,1}+\cdots+d_{2,\,n-1}$ and $-d_2^H=d_{2,\,1}+\cdots+d_{2,\,n}$ on $\bar{A}[1]^{\otimes n}$. The ungraded versions of
\begin{equation}\label{CycBicomplexLemma2}
(-d_2^H)\circ(1-T)\,=\,(1-T)\circ d_2\quad\text{and}\quad d_2\circ N\,=\,N\circ(-d_2^H)
\end{equation}
are proved in \cite[Lemma 2.1.1]{L}, and the key ingredient in the proof is the identity
\begin{equation}\label{CycBicomplexLemma2i}
d_{2,\,i}\circ T_n\,=\,T_n\circ d_{2,\,i-1}\ \text{for}\ 2\leqslant i \leqslant n\quad\text{and}\quad d_{2,\,1}\circ T_n=d_{2,\,n}\,,
\end{equation}
which is also true in our case. Following the same argument as in the proof of \cite[Lemma 2.1.1]{L}, one can prove \eqref{CycBicomplexLemma2}. We leave the detail to the reader.

By comparing \eqref{BarD1} and \eqref{HochD1}, it is clear that $-d_1^H=d_1$ when they restrict to $\bar{A}[1]^{\otimes n}$. A straightforward checking shows that $d_1\circ T_n=T_n\circ d_1$ on $\bar{A}[1]^{\otimes n}$. Therefore, we have
\begin{equation}\label{CycBicomplexLemma1}
(-d_1^H)\circ(1-T)\,=\,(1-T)\circ d_1\quad\text{and}\quad d_1\circ N\,=\,N\circ(-d_1^H)\,.
\end{equation}

Let
$$d_{0,i}(sa_{1},\,\cdots,\,sa_{n})\,=\,(-1)^{|a_{1}|+\cdots+|a_{i}|+i+1}(sa_{1},\,\cdots,\,sa_{i},\,s\Theta,\,\cdots,\,sa_{n})\qquad 0\leqslant i \leqslant n\,,$$
Then $d_0=d_{0,\,0}+\cdots+d_{0,\,n}$ and $-d_0^H=d_{0,\,1}+\cdots+d_{0,\,n}$ on $\bar{A}[1]^{\otimes n}$. The following identity is easy to check
\begin{equation}\label{CycBicomplexLemma0i}
d_{0,\,i}\circ T_n\,=\,T_n\circ d_{0,\,i-1}\ \text{for}\ 1\leqslant i \leqslant n\quad\text{and}\quad d_{0,\,0}=T_n\circ d_{0,\,n}\,.
\end{equation}
Now,
\begin{eqnarray}
(-d_0^H)\circ(1-T)&=& (d_{0,\,1}+\cdots+d_{0,\,n}-T_n\circ d_{0,\,0}-\cdots-T_n\circ d_{0,\,n-1})\\
&=&(d_{0,\,0}+d_{0,\,1}+\cdots+d_{0,\,n}-T_n\circ d_{0,\,0}-\cdots-T_n\circ d_{0,\,n-1}-T_n\circ d_{0,\,n})\nonumber\\
&=&(1-T)\circ d_0\nonumber\,.
\end{eqnarray}
Moreover, \eqref{CycBicomplexLemma0i} implies that
\begin{equation}\label{CycBicomplexLemma0ij}
d_{0,\,i}\circ T_n^j\,=\,
\begin{cases}
T_n^j \circ d_{0,\,i-j} & \text{when\ }0\leqslant j<i\\
T_n^{j+1} \circ d_{0,\,n-j+i} & \text{when\ }i\leqslant j \leqslant n-1
\end{cases}
\end{equation}
Then one can write
\begin{eqnarray*}
d_0\circ N &=& (\sum_{i=0}^{n}d_{0,\,i})\circ (\sum_{j=0}^{n-1}T_n^j)\\
&=&\sum_{0\leqslant j < i \leqslant n} T_n^j\circ d_{0,\,i-j}+\sum_{0\leqslant i \leqslant j \leqslant n-1}T_n^{j+1}\circ d_{0,\,n-j+i}\\
&=&\sum_{1\leqslant q \leqslant n} (\sum_{0\leqslant j \leqslant n-q}T_n^j+\sum_{n-q+1\leqslant j+1 \leqslant n }T_n^{j+1})\circ d_{0,\,q}\\
&=& N\circ (-d_0^H)
\end{eqnarray*}
Therefore, \eqref{CycBicomplexLemma3} holds, and we have proved the desired the lemma.
\end{proof}

Since $(1-T)\circ N=N\circ(1-T)=0$, we have a periodic complex of chain complexes
\begin{equation}\label{PerCycBicomplex}
\cdots\xleftarrow{\ N\ }\,_{\pi}\bar{A}[1]\otimes_{\pi}\rbB(A)\,\xleftarrow{1-T}\,\rbB(A)\,\xleftarrow{\ N\ }\,_{\pi}\bar{A}[1]\otimes_{\pi}\rbB(A)\,\xleftarrow{1-T}\,\rbB(A)\,\xleftarrow{\ N\ }\cdots\,.
\end{equation}
The {\it reduced periodic cyclic bicomplex} of an augmented curved DG algebra $A$ is the bicomplex assocaited to \eqref{PerCycBicomplex}, denoted by $\overline{CC}^{\mathrm{per}}(A)$. The {\it reduced periodic cyclic homology} $\rHC^{\mathrm{per}}_{\bullet}(A)$ of $A$ is (up to a shift in degree) the homology of the {\it direct sum} total complex $\Tot^{\oplus} \big(\overline{CC}^{\mathrm{per}}(A)\big)$, i.e., 
\begin{equation}\label{PHCDef1}
\rHC^{\mathrm{per}}_{\bullet}(A)\,=\,\H_{\bullet+1}[\Tot^{\oplus} \big(\overline{CC}^{\mathrm{per}}(A)\big)]\,.
\end{equation}

\begin{remark}
Notice that \eqref{PerCycBicomplex} is a generaliztion of the periodic complex constructed in \cite[Sec. 5]{Qui89}, which is indentified to be the periodic cyclic bicomplex of a nonunital algebra. In particular, if $A$ is an augmented $k$-algebra (viewed as a curved DG algebra with trivial $\triangledown$ and $\Theta$), then the bicomplex associated to \eqref{PerCycBicomplex} is nothing but the standard reduced periodic cyclic bicomplex for $A$ with the column degrees shifted up by $1$. Therefore, technically, \eqref{PerCycBicomplex} should be $\overline{CC}^{\mathrm{per}}(A)[0,\,1]$ which explains the degree shifting in the definition above.
\end{remark}

Moreover, we have two other complexes
\begin{equation}\label{CycBicomplex}
0\longleftarrow_{\pi}\bar{A}[1]\otimes_{\pi}\rbB(A)\,\xleftarrow{1-T}\,\rbB(A)\,\xleftarrow{\ N\ }\,_{\pi}\bar{A}[1]\otimes_{\pi}\rbB(A)\,\xleftarrow{1-T}\,\rbB(A)\,\xleftarrow{\ N\ }\cdots\,,
\end{equation}
and
\begin{equation}\label{NegCycBicomplex}
\cdots\xleftarrow{\ N\ }\,_{\pi}\bar{A}[1]\otimes_{\pi}\rbB(A)\,\xleftarrow{1-T}\,\rbB(A)\,\xleftarrow{\ N\ }\,_{\pi}\bar{A}[1]\otimes_{\pi}\rbB(A)\,\xleftarrow{1-T}\,\rbB(A)\,\longleftarrow 0\,.
\end{equation}
where the first nonzero column of \eqref{CycBicomplex} is in degree $0$ and the last nonzero column of \eqref{NegCycBicomplex} is in degree $1$.

The bicomplex associated to \eqref{CycBicomplex} is called the {\it reduced cyclic bicomplex} of an augmented curved DG algebra $A$, denoted by $\overline{CC}(A)$. The {\it reduced cyclic homology} $\rHC_{\bullet}(A)$ of $A$ is (up to a shift in degree) the homology of the {\it direct sum} total complex $\Tot^{\oplus}\big(\overline{CC}(A)\big)$, i.e., 
\begin{equation}\label{HCDef1}
\rHC_{\bullet}(A)\,=\,\H_{\bullet+1}[\Tot^{\oplus} \big(\overline{CC}(A)\big)]\,.
\end{equation}
Similarly, the bicomplex associated to \eqref{NegCycBicomplex} is called the  {\it reduced negative cyclic bicomplex} of an augmented curved DG algebra $A$, denoted by $\overline{CC}^{-}(A)$. The {\it reduced negative cyclic homology} $\rHC^{-}_{\bullet}(A)$ of $A$ is (up to a shift in degree) the homology of the {\it direct sum} total complex $\Tot^{\oplus} \big(\overline{CC}^{-}(A)\big)$, i.e., 
\begin{equation}\label{NHCDef1}
\rHC^{-}_{\bullet}(A)\,=\,\H_{\bullet+1}[\Tot^{\oplus} \big(\overline{CC}^{-}(A)\big)]\,.
\end{equation}

It is clear that there is a short exact sequence
\begin{displaymath}
\begin{diagram}
0 & \rTo & \overline{CC}^{-}(A) & \rTo & \overline{CC}^{\mathrm{per}}(A) & \rTo & \overline{CC}(A)[2,\,0] & \rTo & 0\,,
\end{diagram}
\end{displaymath}
which induces the following long exact sequence
\begin{equation}\label{HCLongExactSeq}
\cdots\,\longrightarrow\,\rHC^{-}_{n}(A)\,\longrightarrow\,\rHC^{\mathrm{per}}_{n}(A)\,\longrightarrow\,\rHC_{n-2}(A)\,\longrightarrow\,\rHC^{-}_{n-1}(A)\,\longrightarrow\,\cdots\,.
\end{equation}

By \cite[Thm. 2.1.5]{L}, the rows in $\overline{CC}^{\mathrm{per}}(A)$ are exact. If $\overline{CC}^{\mathrm{per}}(A)$ is concentrated on the upper half-plane (for instance, $A$ is non-negatively graded), the row filtration on $\Tot^{\oplus} \big(\overline{CC}^{\mathrm{per}}(A)\big)$ is bounded below and exhaustive. The classical convergence theorem of spectral sequences forces $\rHC^{\mathrm{per}}_{\bullet}(A)$ to be trivial. Additionally, the spectral sequences associated to the row filtrations on $\Tot^{\oplus} \big(\overline{CC}(A)\big)$ and $\Tot^{\oplus} \big(\overline{CC}^{-}(A)\big)$ collapse at the first pages with only one non-zero column. Therefore, one has
\begin{equation}\label{HCDef2}
\rHC_{\bullet}(A)\,\cong\,\H_{\bullet+1}[\Coker\big(1-T\,:\,\rbB(A)\rightarrow _{\pi}\bar{A}[1]\otimes_{\pi}\rbB(A)\big)]\,,
\end{equation}
and
\begin{equation}\label{NHCDef2}
\rHC^{-}_{\bullet}(A)\,\cong\,\H_{\bullet}[\Ker\big(1-T\,:\,\rbB(A)\rightarrow _{\pi}\bar{A}[1]\otimes_{\pi}\rbB(A)\big)]\,.
\end{equation}
By \cite[Lemma 1.2]{Qui89}, the kernal of $1-T$ is isomorphic to $\rbB(A)^\n$, the cocommutator subspace of $\rbB(A)$. Hence,
\begin{equation}\label{NHCDef3}
\rHC^{-}_{\bullet}(A)\,\cong\,\H_{\bullet}[\rbB(A)^\n]\,.
\end{equation}

\subsection{Cyclic cohomology}
Recall that $\bB(A)$ is weight graded: $\bB(A)=\bigoplus\limits_{n\in \mathbb{N}}\bB(A)^{(n)}$, where $\bB(A)^{(n)}=\bar{A}[1]^{\otimes n}$. We define the bi-graded (weight and homological gradings) linear dual of $\rbB(A)$ to be
$$\big(\rbB(A)\big)_{bi}^\ast\defeq \bigoplus_{n\in \mathbb{N}^+}\big(\rbB(A)^{(n)}\big)^\ast\,.$$
Similarly, the bi-graded linear dual of $_{\pi}\bar{A}[1]\otimes_{\pi}\rbB(A)$ is
$$\big(_{\pi}\bar{A}[1]\otimes_{\pi}\rbB(A)\big)_{bi}^\ast\defeq \bigoplus_{n\in \mathbb{N}}\big(\bar{A}[1]\otimes\rbB(A)^{(n)}\big)^\ast\,.$$

Taking the bi-graded dual of \eqref{PerCycBicomplex}, \eqref{CycBicomplex}, and \eqref{NegCycBicomplex}, one has the following three complexes
\begin{equation}\label{coPerCycBicomplex}
\cdots\xrightarrow{\ N\ }\,\big(_{\pi}\bar{A}[1]\otimes_{\pi}\rbB(A)\big)_{bi}^\ast\,\xrightarrow{1-T^\ast}\,\big(\rbB(A)\big)_{bi}^\ast\,\xrightarrow{\ N^\ast\ }\,\big(_{\pi}\bar{A}[1]\otimes_{\pi}\rbB(A)\big)_{bi}^\ast\,\xrightarrow{1-T^\ast}\,\big(\rbB(A)\big)_{bi}^\ast\,\xrightarrow{\ N^\ast\ }\cdots\,;
\end{equation}
\begin{equation}\label{coCycBicomplex}
0\longrightarrow\,\big(_{\pi}\bar{A}[1]\otimes_{\pi}\rbB(A)\big)_{bi}^\ast\,\xrightarrow{1-T^\ast}\,\big(\rbB(A)\big)_{bi}^\ast\,\xrightarrow{\ N^\ast\ }\,\big(_{\pi}\bar{A}[1]\otimes_{\pi}\rbB(A)\big)_{bi}^\ast\,\xrightarrow{1-T^\ast}\,\big(\rbB(A)\big)_{bi}^\ast\,\xrightarrow{\ N^\ast\ }\cdots\,;
\end{equation}
\begin{equation}\label{coNegCycBicomplex}
\cdots\xrightarrow{\ N\ }\,\big(_{\pi}\bar{A}[1]\otimes_{\pi}\rbB(A)\big)_{bi}^\ast\,\xrightarrow{1-T^\ast}\,\big(\rbB(A)\big)_{bi}^\ast\,\xrightarrow{\ N^\ast\ }\,\big(_{\pi}\bar{A}[1]\otimes_{\pi}\rbB(A)\big)_{bi}^\ast\,\xrightarrow{1-T^\ast}\,\big(\rbB(A)\big)_{bi}^\ast\,\longrightarrow 0\,,
\end{equation}
whose associated bicomplexes are denoted by $\overline{CC}^\ast_{\mathrm{per}}(A)$, $\overline{CC}^\ast(A)$, and $\overline{CC}^\ast_{-}(A)$, respctively.

The {\it reduced periodic cyclic cohomology} $\rHC_{\mathrm{per}}^{\bullet}(A)$, the {\it reduced cyclic cohomology} $\rHC^{\bullet}(A)$, and the {\it reduced negative cyclic cohomology} $\rHC_{-}^{\bullet}(A)$ are (up to shifts in degrees) the homologies of the {\it direct product} total complexes of \eqref{coPerCycBicomplex}, \eqref{coCycBicomplex}, and \eqref{coNegCycBicomplex}, respectively. Explicitly,
\begin{eqnarray}
\rHC_{\mathrm{per}}^n(A) &\defeq& \H_{-n-1}[\Tot^{\prod} \big(\overline{CC}^\ast_{\mathrm{per}}(A)\big)] \label{coPHCDef}\,,\\
\rHC^n(A) &\defeq& \H_{-n-1}[\Tot^{\prod} \big(\overline{CC}^\ast(A)\big)] \label{coHCDef}\,,\\
\rHC_{-}^n(A) &\defeq& \H_{-n-1}[\Tot^{\prod} \big(\overline{CC}^\ast_{-}(A)\big)] \label{coNHCDef}\,.
\end{eqnarray}

Dual to \eqref{HCLongExactSeq}, we have
\begin{equation}\label{coHCLongExactSeq}
\cdots\,\longrightarrow\,\rHC^{n}(A)\,\longrightarrow\,\rHC_{\mathrm{per}}^{n}(A)\,\longrightarrow\,\rHC_{-}^{n+2}(A)\,\longrightarrow\,\rHC^{n+1}(A)\,\longrightarrow\,\cdots\,.
\end{equation}

\subsection{Noncommutative differential forms}\label{subsecNCDiff}
We shall recall some basic definitions and results related to the bimodule of noncommutative differential forms. We refer to the paper of Quillen \cite{Qui89} for a detailed treatment.

Let $R$ be a DG algebra. The bimodule $\Omega^{1}R$ of (noncommutative) differential (1-)forms over $R$ is defined to be the kernel of the multiplication map $\mu\,:\,R\otimes R \rightarrow R$. The $R$-bimodule structure on $\Omega^{1}R$ is induced from the outer bimodule structure on $R\otimes R$. By definition, there is a short exact sequence of $R$-bimodules
\begin{equation}\label{NCDiffDef}
0\longrightarrow \Omega^{1}R \xrightarrow{\ I\ } R\otimes R \xrightarrow{\ \mu\ } R \longrightarrow 0\,.
\end{equation}
Notice that $\Omega^{1}R$ represents the functor $\bDer(R,\,\mbox{--})$, i.e., for any $R$-bimodule $M$, one has 
$$\bDer(R,\,M)\,\cong\,\bHom_{R^{e}}(\Omega^{1}R,\,M)\,,$$
where $\bDer(R,\,M)$ is the complex of all $k$-linear derivations from $R$ to $M$.
If $M=\Omega^{1}R$, the derivation $\partial\,:\,R\rightarrow \Omega^{1}R$, corresponding to the identity map under the isomorphism above, is the universal derivation. 

Given a $R$-bimodule $M$, we denote its commutator quotient space by $M_{\n}$, i.e., $M_{\n}\defeq M/[R,\,M]$, and let $\n\,:\,M\rightarrow M_{\n}$ denote the canonical projection. In the case of a free bimodule, one has the following identification ({\it cf.} \cite[Prop. 3.5]{Qui89}).
\begin{proposition}\label{FreeBimodN}
There is a canonical isomorphism $(R\otimes V\otimes R)_{\n} \cong V\otimes R$, and the canonical projection is given by the composition
$$R\otimes V\otimes R \xrightarrow{\ t\ } V\otimes R\otimes R \xrightarrow{\id\otimes \mu} V\otimes R\,,$$
where $t$ is the cyclic permutation.
\end{proposition}

We can apply the commmutator quotient space functor to the canonical injection $I\,:\,\Omega^{1}R\rightarrow R\otimes R$, using Proposition \ref{FreeBimodN} to identify the commutator quotient space for free bimodule. This gives the following commutative diagram
\begin{displaymath}
\begin{diagram}[width=1.2cm,height=0.9cm]
\Omega^{1}R & \rTo^{I} & R\otimes R\\
\dTo^{\n} & & \dTo_{\mu\circ t}\\
\Omega^{1}R_{\n} & \rTo^{\beta} & R
\end{diagram}
\end{displaymath}
where $\beta$ is the unique map such that the diagram is commutative.

We define $\bar{\partial}\,:\, R\rightarrow \Omega^{1}R_{\n}$ to be the composition $\n \circ \partial$. It can be shown that both $\beta\circ \bar{\partial}$ and $\bar{\partial} \circ \beta$ are zero. Hence one has the following periodic complex
\begin{equation}\label{NCDiffPeriodic1}
\cdots \xleftarrow{\ \bar{\partial}\ } R \xleftarrow{\ \beta\ } \Omega^{1}R_{\n} \xleftarrow{\ \bar{\partial}\ } R \xleftarrow{\ \beta\ }  \Omega^{1}R_{\n} \xleftarrow{\ \bar{\partial}\ } \cdots\,.
\end{equation}

\begin{example}\label{NCDiffSemiFree}\cite[Example 3.10]{Qui89}
Let $V$ be a graded vector space and $R$ be the tensor algebra $T(V)$. Then $\Omega^{1}R$ can be identified with $R\otimes V\otimes R$, and the isomorphism $I\,:\, R\otimes V \otimes R \rightarrow \Omega_{R}^{1}\subseteq R\otimes R$ is given by
\begin{eqnarray}
&& I\{(v_{1},\,\cdots,\,v_{p-1})\otimes v_{p}\otimes (v_{p+1},\,\cdots,\,v_{m})\}\\
&=& (v_{1},\,\cdots,\,v_{p})\otimes(v_{p+1},\,\cdots,\,v_{m})-(v_{1},\,\cdots,\,v_{p-1})\otimes(v_{p},\,\cdots,\,v_{m})\,.\nonumber
\end{eqnarray}
Under this isomorphism, the universal derivation $\partial\,:\,R \rightarrow \Omega^{1}R$  is given by
$$\partial(v_{1},\,\cdots,\,v_{n})=\sum_{i=1}^{n}(v_{1},\,\cdots,\,v_{i-1})\otimes v_{i}\otimes (v_{i+1},\,\cdots,\,v_{n})\,,$$
and for all $r\in R$,
$$ I\circ \partial (r)= r\otimes 1 - 1\otimes r \,.$$
Moreover, by Proposition \ref{FreeBimodN}, one can identify  $\Omega^{1}R_{\n}$ with $V\otimes R$ in such a way that the canonical maps $\bar{\partial}$ (it may be called the {\it cyclic derivative}) and $\beta$ are given by the formulas
\begin{eqnarray}\label{CycDer}
&& \bar{\partial}(v_{1},\,\cdots,\,v_{n})\\
&=&\sum_{i=1}^{n}(-1)^{(|v_{1}|+\cdots+|v_{i-1}|)(|v_{i}|+\cdots+|v_{n}|)}v_{i}\otimes(v_{i+1},\,\cdots,\,v_{n},\,v_{1},\,\cdots,\,v_{i-1})\,.\nonumber\\
\nonumber\\
&& \beta \{v_{1}\otimes(v_{2},\,\cdots,\,v_{m})\}=[v_{1},\,(v_{2},\,\cdots,\,v_{m})]\,.
\end{eqnarray}
\end{example}

\begin{remark}
It can be shown that, in the case of the example above, the periodic complex \eqref{NCDiffPeriodic1} is exact provided one replaces $R$ by $\bar{R}\defeq R/k$ (see \cite[Sec. 3]{Kassel88} for a closely related discussion).
\end{remark}

\subsection{Cyclic cohomology of Koszul QLC algebra}\label{subsecFTHCLongExactSeq}
Let us recall a theorem due to Feigin and Tsygan \cite[Thm. 1]{FT85} which essentially says that the reduced cyclic homology can be viewed as a (non-abelian) derived functor in the sense of Quillen's homotopical algebra \cite{Qui67}. For a conceptual proof of this theorem, we refer to \cite[Sec. 3]{BKR}.

\begin{theorem}\label{FTrHC}
Given a graded $k$-algebra $A$. Let $R\xrightarrow{\sim} A$ be a semi-free resolution of $A$ in the category of DG algebras (i.e., the underlying graded algebra of $R$ is free). Then there is a natural isomorphism of graded vector spaces 
$$\rHC_\bullet(A)\,\cong\,\H_\bullet(R_\n)\,,$$
where $R_\n \defeq R/(k+[R,\,R])$. 
\end{theorem}

Pick a non-negatively graded semi-free resolution $R$ of $A$ (for example, the cobar-bar construction if $A$ is non-negatively graded). Recall from Example \ref{NCDiffSemiFree} and the remark after it, there is a periodic complex 
\begin{equation}\label{NCDiffPeriodic2}
\cdots \xleftarrow{\ \bar{\partial}\ } \bar{R} \xleftarrow{\ \beta\ } \Omega^{1}R_{\n} \xleftarrow{\ \bar{\partial}\ } \bar{R} \xleftarrow{\ \beta\ }  \Omega^{1}R_{\n} \xleftarrow{\ \bar{\partial}\ } \cdots\,,
\end{equation}
and the rows of the associated bicomplex are exact. Truncating the bicomplex associated to \eqref{NCDiffPeriodic2} at the $\bar{R}$-column in degree $0$, we obtain the following first quadrant bicomplex
\begin{equation}\label{XPlusComplex}
X^+(R)\defeq[0 \longleftarrow \bar{R} \xleftarrow{\ \beta\ } \Omega^{1}R_{\n} \xleftarrow{\ \bar{\partial}\ } \bar{R} \xleftarrow{\ \beta\ }  \Omega^{1}R_{\n} \xleftarrow{\ \bar{\partial}\ } \cdots]\,.
\end{equation}
It can be shown that (see \cite[Sec. 5.4]{BKR} for details)
\begin{equation}\label{XPlusHC}
\H_n[\Tot X^+(R)]\,\cong\,\H_\bullet(R_\n)\,\cong\, \rHC_n(A)\,.
\end{equation}

The following theorem is a generalization of \cite[Thm. 2.4.1 (a)]{FT87} to the case of Koszul QLC algebras.
\begin{theorem}\label{ThmFTHCLongExactSeq}
Let $A=A(V,\,R)$ be a QLC algebra, where $V$ is a non-negatively graded locally finite graded vector space, and $(qA)^!$ be its Koszul dual curved DG algebra. Suppose $qA$ is Koszul, then there is a long exact sequence
\begin{equation}\label{FTHCLongExactSeq}
\cdots\,\longrightarrow\,\rHC_{\mathrm{per}}^{-2-n}((qA)^!)\,\longrightarrow\,\rHC_{n}(A)\,\longrightarrow\,\rHC^{-1-n}((qA)^!)\,\longrightarrow\,\rHC_{\mathrm{per}}^{-1-n}((qA)^!)\,\longrightarrow\,\cdots\,.
\end{equation}
\end{theorem}

\begin{proof}
Since $qA$ is Koszul, by Theorem \ref{KoszulFreeRes}, we can choose $R$ to be the cobar construction $\cb((qA)^\dual)$ of the Koszul dual curved coalgebra $(qA)^\dual$. Let $V=\overline{(qA)^\dual}[-1]$, then $R\cong T(V)$ as graded algebras. Thus, the underlying graded vector spaces of $\Omega^{1}R$ is isomorphic to $R\otimes V\otimes R$. Notice that the following composition is identity,
$$R\otimes V\otimes R \xrightarrow{\ i\ } R\otimes R\otimes R \xrightarrow{\ \tilde{\partial}\ } R\otimes V\otimes R\cong \Omega^{1}R\,,$$
where $i$ is the obvious inclusion and $\tilde{\partial}$ is the unique $R$-bimodule morphism extended from the universal derivation $\partial$. Moreover, because $\tilde{\partial}$ is a morphism of complexes, one has the following commutative diagram
\begin{displaymath}
\begin{diagram}[width=1.2cm,height=0.9cm]
R\otimes V\otimes R & \rTo^{\ \ \id} & \Omega^{1}R\\
\dTo^{d_{R\otimes R\otimes R}} & & \dTo_{d_{\Omega^{1}R}}\\
R\otimes R\otimes R & \rTo^{\ \ \tilde{\partial}} & \Omega^{1}R
\end{diagram}
\end{displaymath}
Thus, the differential on $\Omega^{1}R$ is given by the following formula: for any $p\otimes v \otimes q\in R\otimes V\otimes R$, with $v=s^{-1}c$, one has
\begin{eqnarray}\label{RVRDiff}
&&d_{\Omega^{1}R}(p\otimes v\otimes q)\\ 
&=&d_{R}(p)\otimes v\otimes q +(-1)^{|p|}p\otimes d_{V}(v)\otimes q + (-1)^{|p|+|v|} p\otimes v\otimes d_{R}(q)\nonumber \\
&&+(-1)^{|p|+|v'|}p\cdot v'\otimes v''\otimes q + (-1)^{|p|+|v'|}p\otimes v'\otimes v''\cdot q\,.\nonumber
\end{eqnarray}
Here $d_{V}(v)=-s^{-1}d_{C}(c)$, and $\bar{\Delta}(sv)=s(v')\otimes s(v'')$ in the Sweedler notation.

Moreover, since $\n\,:\,\Omega^{1}R\rightarrow\Omega^{1}R_{\n}$ is a morphism of complexes, the differential on $\Omega^{1}R_{\n}$ is given by the following formula:
\begin{eqnarray}\label{VRDiff}
&&d_{\Omega^{1}R_{\n}}(v\otimes q)\\ 
&=&d_{V}(v)\otimes q + (-1)^{|v|} v\otimes d_{R}(q)\nonumber \\
&&+(-1)^{|v'|(1+|v''|+|q|)}v''\otimes q\cdot v' + (-1)^{|v'|}v'\otimes v''\cdot q\,.\nonumber
\end{eqnarray}

Since $V$ is non-negatively graded and locally finite, the weight decomposition \eqref{KDCWeightDecomp} implies that $(qA)^\dual$ is a non-negatively graded, connected, locally finite graded coalgerba. As the graded dual of $(qA)^\dual$, $(qA)^!$ is a non-positively graded, connected, locally finite graded algerba. Hence, $\big(\overline{(qA)^!}[1]\big)^{\otimes n}$ is non-positively graded and locally finite for each $n$, and this implies that its graded linear dual is isomorphic to $\big(\overline{(qA)^\dual}[-1]\big)^{\otimes n}$. Hence, as graded vector spaces,  
$$\big(\rbB((qA)^!)\big)_{bi}^\ast\,\cong\,\bar{R} \qquad\text{and}\qquad \big(_{\pi}(\overline{(qA)^!}[1]\otimes_{\pi}\rbB((qA)^!)\big)_{bi}^\ast\,\cong\,V\otimes R\,.$$
Moreover, it is easy to see that the differential on $\rbB((qA)^!)$ is linearly dual to the differential on $\bar{R}$ and the differential on $\overline{(qA)^!}[1]\otimes_{\pi}\rbB((qA)^!)$ is linearly dual to \eqref{VRDiff}. Therefore, the bicomplex associated to $\overline{CC}^\ast_{-}((qA)^!)$ is isomorphic to $X^+(R)$, and
$$\rHC_n(A)\,\cong\,\H_n[\Tot X^+(R)]\,\cong\,\H_{n-1}[\Tot^{\prod} \big(\overline{CC}^\ast_{-}((qA)^!)\big)]\,=\,\rHC_{-}^{-n}((qA)^!)\,,$$
where the degree shifting in the second isomorphism is due to the first column of the bicomplex associated to $\overline{CC}^\ast_{-}((qA)^!)$ is in degree $-1$. The desired long exact sequence \eqref{FTHCLongExactSeq} is obtained by replacing $\rHC_{-}^{-n}((qA)^!)$ in the long exact sequence \eqref{coHCLongExactSeq} with $\rHC_{n}(A)$.
\end{proof}

\begin{remark}
Notice that our definitions of cyclic cohomologies are slight different from the ones given in \cite{FT87}. Specifically, in \cite{FT87} the periodic cyclic cohomology and the cyclic cohomology are the homology of the {\it direct sum} total complexes of $\overline{CC}^\ast_{\mathrm{per}}(A)$ and $\overline{CC}^\ast(A)$, respectively. We denote the resulting cohomology theories by $\rHC_{\mathrm{per},\,\oplus}^{\bullet}$ and $\rHC_{\oplus}^{\bullet}$, respectively. The long exact sequence \eqref{FTHCLongExactSeq} is still valid if one replaces $\rHC_{\mathrm{per}}^{\bullet}$ and $\rHC^{\bullet}(A)$ by $\rHC_{\mathrm{per},\,\oplus}^{\bullet}$ and $\rHC_{\oplus}^{\bullet}(A)$, respectively. Since $(qA)^!$ is non-positively graded and connected, it is clear that the bicomplex $\overline{CC}^\ast_{\mathrm{per}}((qA)^!)$ is concentrated on the upper half-plane. Since its rows are exact, using the standard argument of spectral sequence, one can show that $\rHC_{\mathrm{per},\,\oplus}^{\bullet}((qA)^!)$ is trivial. Hence, we obtain \cite[Thm. 2.4.1 (b)]{FT87}, namely,
\begin{equation}\label{FTHCIso}
\rHC_{n}(A)\,\cong\,\rHC_{\oplus}^{-1-n}((qA)^!)\,.
\end{equation}
\end{remark}

\section{Commutative QLC Koszul algebras}\label{secComQLC}
We discuss commutative QLC Koszul algebras in this section, which can be viewed as a commutative/Lie analog to Section \ref{secQLC}. We shall freely use the language of operads and their Koszul duality theories. We refer to the book of Loday and Vallette \cite{LV} for a complete exposition. 
\subsection{Curved DG Lie coalgebras}
Recall from \cite{Michaelis80}, a graded Lie coalgebra $\mfG$ over $k$ is a graded vector space $\mfG$ equipped with a degree $0$ linear map $]\, \mbox{--}\, [ \,:\,\mfG \rightarrow \mfG\otimes\mfG$ satisfying
\begin{equation}\label{Liecoalg1}
\tau_{(12)}\circ\ ]\, \mbox{--}\, [\ =-\  ]\, \mbox{--}\, [\ \,,
\end{equation}
and
\begin{equation}\label{Liecoalg2}
\big(\id+\tau_{(123)}+\tau_{(123)}^2\big)\circ\big(\ ]\, \mbox{--}\, [\ \otimes \id\big)\circ\ ]\, \mbox{--}\, [\ = 0\,,
\end{equation}
where $\tau_{(12)}$ and $\tau_{(123)}$ denote the cyclic permutations on the graded vector spaces $\mfG\otimes \mfG$ and $\mfG\otimes \mfG\otimes \mfG$, respectively. The linear map $]\, \mbox{--}\, [ $ is called the {\it Lie cobracket} of $\mfG$ and the identity \eqref{Liecoalg2} is called the {\it co-Jacobi identity}.

\begin{remark}
Because of \eqref{Liecoalg1}, the Lie cobracket factors through $\wedge^2\mfG \hookrightarrow \mfG\otimes \mfG$, where the inclusion is given by:
$$x\wedge y\,\mapsto\,x\otimes y-(-1)^{|x||y|}y\otimes x\,.$$
\end{remark}

A {\it curved DG Lie coalgebra} over $k$ is a $k$-graded Lie coalgebra $\mfG$ endowed with a Lie coderivation $d$ (i.e., $\ ]\, \mbox{--}\, [\ \circ d=(d\otimes \id+\id\otimes d)\circ \ ]\, \mbox{--}\, [\ $) of degree $-1$ and a linear map $h\,:\,\mfG\rightarrow k$ of degree $-2$ satisfying
\begin{enumerate}
	\item $d^2\,=\, (h\otimes \id)\circ\ ]\, \mbox{--}\, [\ $
	\item $h\circ d\,=\,0$
\end{enumerate}

Recall that there is a Lie coalgebra functor ${\mathcal Lie}^{c}$ from $\GrCoAlg_k$, the category of graded coalgebras over $k$, to $\GrCoLAlg_k$, the category of graded Lie coalgebras over $k$: For each $(C,\,\Delta)\in \GrCoAlg_k$, the underlying graded vector space of ${\mathcal Lie}^{c}(C)$ is identical to $C$ and the Lie cobracket is given by $\Delta-\tau_{(12)}\circ \Delta$. 

\begin{proposition}
Given a curved DG coalgebra $(C,\,d,\,h)$, then $({\mathcal Lie}^{c}(C),\,d,\,h)$ is a curved DG Lie coalgebra.
\end{proposition}

\begin{proof}
The coderivation $d$ becomes a Lie coderivation in ${\mathcal Lie}^{c}(C)$. Indeed,
\begin{eqnarray*}
\ ]\, \mbox{--}\, [\ \circ d &=& (\Delta-\tau_{(12)}\circ \Delta)\circ d\\
&=& (\id-\tau_{(12)})\circ(d\otimes \id+\id\otimes d)\circ\Delta\\
&=& (d\otimes \id+\id\otimes d)\circ\Delta-(\id\otimes d+d\otimes \id)\circ \tau_{(12)}\circ \Delta\\
&=& (d\otimes \id+\id\otimes d)\circ \ ]\, \mbox{--}\, [\ \,.
\end{eqnarray*}
Moreover,
$$d^2\,=\, (h\otimes \id- \id\otimes h)\circ\Delta\,=\,(h\otimes \id)\circ (\Delta-\tau_{(12)}\circ \Delta)\,=\,(h\otimes \id)\circ\ ]\, \mbox{--}\, [\ \,.$$
Finally, $h\circ d=0$ by the definition of a curved DG coalgebra.
\end{proof}

The {\it Lie cobar construction} $\cb_{\mathtt{Lie}}(\mfG)$ of a curved DG Lie coalgebra $(\mfG,\,d,\,h)$ is defined as the graded symmetric algebra $\bSym(\mfG[-1])$ equipped with three derivations $d_{0}$, $d_{1}$, and $d_{2}$, where $d_{0}$ (resp., $d_{1}$) is induced by $h$ (resp., $d$) on $\mfG$, and $d_{2}|_{\mfG[-1]}$ is given by the composite map
$$\mfG[-1] \cong k[-1] \otimes \mfG\xrightarrow{\Delta_{-1}\otimes\ ]\, \mbox{--}\, [\ } k[-1] \otimes k[-1] \otimes \wedge^2\mfG\,\cong\,\bSym^{2}(\mfG[-1])\,.$$
Explicitly, given $s^{-1}\xi\in \mfG[-1]$,
$$d_{2}(s^{-1}\xi)=-\frac{1}{2}\sum_{i}(-1)^{|\xi_{i}'|}s^{-1}\xi'_{i}\cdot s^{-1}\xi''_{i}\,,$$
where $]\xi[=\sum\limits_{i}\xi'_{i}\otimes \xi''_{i}$.

The following proposition is analogous to Proposition \ref{CobarDGA}.
\begin{proposition}\label{LieCobarDGCA}
Given a curved DG Lie coalgebra $(\mfG,\,d,\,h)$, the Lie cobar construction $\cb_{\mathtt{Lie}}(\mfG)$ is a semi-augmented commutative DG algebra with differential given by $d_0+d_1+d_2$.
\end{proposition}

\begin{proof}
Since $d_0$, $d_1$, and $d_2$ are derivations of degree $-1$,  $(d_0+d_1+d_2)^2=\frac{1}{2}[d_0+d_1+d_2,\,d_0+d_1+d_2]$. It suffices to show that for any $s^{-1}\xi\in \mfG[-1]$,  
\begin{eqnarray*}
&& (d_0+d_1+d_2)^2(s^{-1}\xi)\\
&=& (d_0d_0+d_0d_1+d_0d_2+d_1d_0+d_1d_1+d_1d_2+d_2d_0+d_2d_1+d_2d_2)(s^{-1}\xi)\\
&=& 0\,.
\end{eqnarray*}
Indeed,
\begin{itemize}
    \item $(d_0d_0)(s^{-1}\xi)=(d_1d_0)(s^{-1}\xi)=(d_2d_0)(s^{-1}\xi)=0$ for the restrictions of $d_0$, $d_1$, and $d_2$ to $k$ vanish.
    \item $(d_0d_1)(s^{-1}\xi)=0$ for $h\circ d\,=\,0$;
    \item $(d_0d_2+d_1d_1)(s^{-1}\xi)=0$ since $d^2\,=\, (h\otimes \id)\circ\ ]\, \mbox{--}\, [\ $;
    \item $(d_1d_2+d_2d_1)(s^{-1}\xi)=0$ since $d$ is a coderivation;
    \item $(d_2d_2)(s^{-1}\xi)=0$ because of the anticommutativity of $\mfG$, the commutativity of $\bSym(\mfG[-1])$, and the co-Jacobi identity on $\mfG$.
\end{itemize}
\end{proof}

\subsection{Commutative QLC algebras}
A {\it commutative quadratic-linear-constant (QLC)} data is a vector space $V$ (concentrated in degree $0$) together with a subspace
$$R \subset k\oplus V\oplus \Sym^2(V)\,.$$ The commutative QLC algebra $A=A(V,\,R)\defeq \Sym(V)/(R)$ is the quotient of the free commutative algebra over $V$ by the ideal $(R)$ generated by $R$. We assume that $R$ satisfies the following two properties:
\begin{eqnarray}
&R\cap (k\oplus V)=\{0\}&\label{cqlca1}\\
&(R)\cap (k\oplus V\oplus \Sym^2(V))=R\,.&\label{cqlca2}
\end{eqnarray}
The first condition amounts to the minimality of the space of generators of $A(V,\,R)$, and the second condition implies that $(R)$ does not create new QLC relations.

Let $q\,:\,\Sym(V)\twoheadrightarrow  \Sym^2(V)$ be the projection onto the quadratic part of symmetric algebra, and $qR$ be the image of $R$ under $q$. \eqref{cqlca1} guarantees that there exists a linear map $\psi\,:\,qR\rightarrow k\oplus V$ such that $R$ is the graph of $\psi$. Write $\psi=(-\phi,\,\theta)$, where $\phi\,:\,qR\rightarrow V$ and $\theta\,:\,qR\rightarrow k$. One has
\begin{equation}\label{CQLCgraph}
R\,=\,\{x-\phi(x)+\theta(x)\,\mid\,x\in qR\}\,.
\end{equation}

Let $\mathit{Com}$ denote the binary quadratic operad encoding {\it nonunital} commutative algebras. The notion of quadratic algebras over a binary quadratic operad was introduced by Ginzburg and Kapranov in \cite{GK94}. In \cite{JBMilles12}, Mill\`{e}s extends it to the notion of monogenic algebras over a quadratic operad. We denote by $qA$ the associated quadratic $\mathit{Com}$-algebra of $A$, i.e., $qA=\overline{\Sym(V)}/(qR)$, where $\overline{\Sym(V)}\defeq \Sym(V)/k$. Notice that $qA$ is weight graded (see \cite[Prop. 4.1]{JBMilles12}):
\begin{equation}\label{CKAWeightDecomp}
qA\,=\,\bigoplus_{n\in \mathbb{N}^+}qA^{(n)}\,=\,V \oplus (\Sym^2(V)/qR) \oplus \cdots \oplus \big(\Sym^n(V)/\Sym^{n-2}(V)qR\big)\oplus \cdots\,.
\end{equation}

Recall that $\mathit{Com}$ is a Koszul operad, and the operadic desuspension of its Koszul dual cooperad $\mathit{S}^{-1}\mathit{Com}^\dual$ is $\mathit{Lie}^c$, the cooperad encodes Lie coalgebras. According to \cite{JBMilles12}, one can associate a graded Lie coalgebra $(qA)^\dual$ to the quadratic $\mathit{Com}$-algebra $qA$, and it is called {\it Koszul dual $\mathit{Lie}^c$-coalgebra} of $qA$. It is a positively graded Lie coalgebra with a weight decomposition given by (see \cite[Prop. 4.3]{JBMilles12})
\begin{equation}\label{KDLieCWeightDecomp}
(qA)^{\dual}\,=\,\bigoplus_{n\in \mathbb{N}^+}(qA)^{\dual,\,(n)}\,,
\end{equation}
where $(qA)^{\dual,\,(1)}=sV$, $(qA)^{\dual,\,(2)}=s^2qR$, and
$$(qA)^{\dual,\,(n)}\,=\,\bigcap_{i+2+j=n} \mathit{Lie}^c(n-1)\otimes_{\mathbb{S}_{n-1}}\big((sV)^{\otimes i}\otimes s^2qR \otimes (sV)^{\otimes j}\big)\quad\text{for\ }n\geqslant 3\,.$$

Then, $\phi$ induces a degree $-1$ map $\tilde{\phi}$ from $(qA)^{\dual}$ to $sV$ via the composition:
\begin{equation}\label{phiqadual2}
\tilde{\phi}\,:\,(qA)^{\dual}\twoheadrightarrow s^2qR\xrightarrow{s^{-1}\phi} sV\,,
\end{equation}
and $\theta$ induces a degree $-2$ map $\tilde{\theta}$ from $(qA)^{\dual}$ to $k$ via the composition:
\begin{equation}\label{thetaqadual2}
\tilde{\theta}\,:\,(qA)^{\dual}\twoheadrightarrow s^2qR\xrightarrow{s^{-2}\theta} k\,.
\end{equation}
In \cite{Idrissi18}, Idrissi extends Lemma \ref{PP05LemmaCC} to any binary quadratic operad. In particular, in the case of operad $\mathit{Com}$ and cooperad $\mathit{Lie}^c$, we obtain the following proposition
\begin{proposition}\cite[Prop. 3.4]{Idrissi18}\label{Idrissi18LemmaCC} 
Given $Y\in (qA)^{\dual,\,(3)}$, the maps $\tilde{\phi}$ and $\tilde{\theta}$ satisfy the following conditions
\begin{eqnarray}
&\tilde{\phi}(Y)\,\in\,s^2qR& \label{Idrissi18LemmaCC1}\\
&\tilde{\phi}^2(Y)\,=\,(h\otimes \id)\circ\ ] Y [\ &\label{Idrissi18LemmaCC2}\\
&\tilde{\theta}\circ\tilde{\phi}(Y)\,=\,0\,.&\label{Idrissi18LemmaCC3}
\end{eqnarray}
\end{proposition}

The following proposition is a Lie analog of Proposition \ref{KoszulCurCAg}, and the proof is rely on Proposition \ref{Idrissi18LemmaCC}.
\begin{proposition}\label{KoszulCurLieCAg}
$((qA)^\dual,\,\tilde{\phi},\,\tilde{\theta})$ is a curved DG Lie coalgebra.
\end{proposition}
We call this curved DG Lie coalgebra the {\it Koszul dual curved Lie coalgebra} of $A$. 

\subsection{The Lie cobar construction of the Koszul dual curved Lie coalgebra}
Consider the following composite map of degree $-1$
$$\kappa\,:\,(qA)^{\dual}\twoheadrightarrow sV \xrightarrow{s^{-1}}V\hookrightarrow A\,.$$
Recall that the Lie cobar construction $\cb_{\mathtt{Lie}}((qA)^{\dual})$ is a commutative DG algebra whose underlying graded commutative algebra is $\bSym((qA)^{\dual}[-1])$. Hence, $\kappa$ induces an algebra morphism $g_\kappa$ from $\cb_{\mathtt{Lie}}((qA)^{\dual})$ to $A$. Moreover, it can be checked that $g_\kappa$ is compatible with differentials, i.e., $g_\kappa\circ (d_0+d_1+d_2)=0$. One replaces the codomain of $\kappa$ with $qA$ and gets another map of degree $-1$
$$\kappa'\,:\,(qA)^{\dual}\twoheadrightarrow sV \xrightarrow{s^{-1}}V\hookrightarrow qA\,.$$
However, as the domain of $\kappa'$, $(qA)^{\dual}$ is viewed only as a graded Lie coalgebra (with trivial $d$ and $h$). Similarly, one can check that $\kappa'$ induces a DG algebra morphism $g_{\kappa'}\,:\,(\cb_{\mathtt{Lie}}((qA)^{\dual}),\,d_2)\rightarrow k\oplus qA$.

Let $\alpha\,:\,\mathit{Com}^\dual\rightarrow \mathit{Com}$ be the operadic twisting morphism. There is a functor $\cb_\alpha$ from the category of DG $\mathit{Com}^\dual$-coalgebras (or equivalently, DG $\mathit{Lie}^c$-coalgebras) to the category of DG $\mathit{Com}$-algebras. Recall that a $\mathit{Com}^\dual$-coalgebra structure on $s^{-1}\mfG$ is equivalent to a $\mathit{Lie}^c$-coalgebra structure on $\mfG$. Given a DG Lie coalgebra $\mfG$, the Lie cobar construction $\cb_{\mathtt{Lie}}(\mfG)$ is an augmented commutative algebra, and one has the following isomorphism ({\it cf.} \cite[Prop. 11.2.5]{LV})
$$\cb_{\mathtt{Lie}}(\mfG)\,\cong\,k\oplus \cb_\alpha(\mfG)\,.$$
Recall that from \cite[Thm. 4.9]{JBMilles12} the quadratic $\mathit{Com}$-algebra $qA$ is {\it Koszul} if the canonical projection $\cb_\alpha((qA)^\dual)\twoheadrightarrow qA$ is a quasi-isomorphism of DG $\mathit{Com}$-algebras (hence, $k\oplus \cb_\alpha((qA)^\dual)$ is quasi-isomorphic to $k\oplus qA$). The following theorem is a Lie analog of Theorem \ref{KoszulFreeRes}.

\begin{theorem}\label{CKoszulFreeRes}
Let $A$ be a commutative QLC algebra, and $\mfG=((qA)^\dual,\,\tilde{\phi},\,\tilde{\theta})$ be its Koszul dual curved Lie coalgebra. Suppose $qA$ is Koszul, then the morphism of commutative DG algebras $g_\kappa\,:\,\cb_{\mathtt{Lie}}(\mfG)\rightarrow A$ is a quasi-isomorphism.
\end{theorem}

\begin{proof}
Since $\mfG$ is weight graded, the graded commutative algebra $\bSym(\mfG[-1])$ is weight graded as well. Consider the filtration $F_r$ of $\cb_{\mathtt{Lie}}(\mfG)$ defined by its weights. The three components of the differential map $d_{\cb_{\mathtt{Lie}}(\mfG)}=d_0+d_1+d_2$ satisfy
$$d_2\,:\,F_r\rightarrow F_r \qquad d_1\,:\,F_r\rightarrow F_{r-1} \qquad d_0\,:\,F_r\rightarrow F_{r-2}\,.$$
Thus, the filtration $F_r$ is stable under $d_{\cb_{\mathtt{Lie}}(\mfG)}$. Since it is bounded below and exhaustive, the associated spectral sequence $E_{rs}^{\bullet}$ converges to the homology $\H_\bullet(\cb_{\mathtt{Lie}}(\mfG))$.

$E_{rs}^{0}$ is made up of the elements in $\cb_{\mathtt{Lie}}(\mfG)$ of degree equal to $r+s$ and weight equal to $r$, and the differential $d^0$ is given by $d_2$. The Koszulness of $qA$ implies that  $g_{\kappa'}\,:\,(\cb_{\mathtt{Lie}}((qA)^{\dual}),\,d_2)\rightarrow k\oplus qA$ is a quasi-isomorphism. Thus, $\H_0(\cb_{\mathtt{Lie}}(\mfG),\,d_2)=k\oplus qA$, and $\H_{i}(\cb_{\mathtt{Lie}}(\mfG),\,d_2)=0$ for $i>0$. It follows that $E^1_{rs}\neq 0$ only if $r+s=0$. Therefore, the spectral sequence stables at page $1$, and one has $\H_{i}(\cb_{\mathtt{Lie}}(\mfG))=0$ for $i>0$. 

The following diagram presents the weight decomposition of  $\cb_{\mathtt{Lie}}(\mfG)$ at degree $1$ and $0$:
\begin{displaymath}
\begin{diagram}[height=0.75cm,width=3cm]
& & & & \text{Weights}\\
& \Sym^2(V)\cdot sqR & \rTo^{d_2\qquad\quad}  & \Sym^4(V) & (4)\\
& & \rdTo^{d_1}\rdTo(2,4)^{d_0}& & \\
& V\cdot sqR & \rTo^{d_2\qquad\qquad\qquad} & \Sym^3(V) & (3)\\
& & \rdTo^{d_1}\rdTo(2,4)^{d_0}& & \\
& sqR & \rTo^{d_2\qquad\qquad\qquad}  & \Sym^2(V) & (2)\\ 
& & \rdTo^{d_1}\rdTo(2,4)^{d_0}& & \\
& & & V & (1)\\
& && & \\
&  & & k & (0)\\
\text{Degrees} & 1 & & 0 &
\end{diagram}
\end{displaymath}
where the differential $d_2$ is simply the inclusion, and $d_1$ (resp., $d_0$) is induced by $\phi$ (resp., $\theta$). By the construction, $g_\kappa$ sends $\cb_{\mathtt{Lie}}(\mfG)_{0}$ isomorphically to $\Sym(V)$. Moreover, it is clear that this isomorphism induces a bijection between the boundaries in $\cb_{\mathtt{Lie}}(\mfG)_{0}$ and $(R)$. Therefore, $g_\kappa$ sends $\H_{0}(\cb_{\mathtt{Lie}}(\mfG))$ isomorphically to $A$. All in all, we have shown that $g_\kappa$ induces an isomorphism between the homologies of $\cb_{\mathtt{Lie}}(\mfG)$ and $A$.
\end{proof}

\subsection{Relationship with the QLC Koszul algebras}
Given a commutative QLC algebra $A=A(V,\,R)$, let $qA$ be the associated quadratic $\mathit{Com}$-algebra and $(qA)^\dual$ be the Koszul dual curved Lie coalgebra. If $V$ is finite dimensional, it can be shown that the graded linear dual of $(qA)^\dual$ is a curved DG Lie algebra, called the {\it Koszul dual curved Lie algebra} of $A$ and denoted by $(qA)^!$. The following proposition characterizes $(qA)^!$ in terms of quadratic Lie algebra.

\begin{proposition}
The underlying graded Lie algebra of $(qA)^!$ is isomorphic to
$$\mathit{Lie}(s^{-1}V^\ast)/(s^{-2}(qR)^\perp)\,,$$
where $(qR)^\perp$ is the annihilator of $qR$ for the natural pairing $\Sym^2(V^\ast)\otimes \Sym^2(V)\rightarrow k$.
\end{proposition}

\begin{proof}
By the construction of $(qA)^\dual$ in \cite{JBMilles12}, it is universal among the Lie subcoalgebras $\mfG$ of the cofree Lie coalgebra $\mathit{Lie}^c(sV)$ such that the composite 
\begin{equation}\label{SESKDLiecoalg}
\mfG\,\hookrightarrow\,\mathit{Lie}^c(sV) \,\twoheadrightarrow\,\mathit{Lie}^c(2)(sV)/s^2qR 
\end{equation}
is zero, where $s^2qR$ is viewed as a subspace of $\mathit{Lie}^c(2)(sV)$ via the idrentification $s^2(\Sym^2(V))\cong \wedge^2(sV)$.  The word ``universal" means that for any such Lie subcoalgebras $\mfG$, there exists a unique Lie coalgebra morphism $\mfG \rightarrow (qA)^\dual$ such that the following diagram commutes:
\begin{displaymath}
\begin{diagram}[height=0.7cm]
& & (qA)^\dual & &\\
& \ruTo & & \rdTo&\\
\mfG & & \rTo & & L^c(sV)\,.
\end{diagram}
\end{displaymath}

Notice that $\mathit{Lie}=(\mathit{Lie}^c)^\ast$, where the linear dualization means the ``arity-graded linear dualization'', i.e., $\mathit{Lie}(n)=(\mathit{Lie}^c(n))^\ast$. Since $V$ is finite dimensional, the linear dual of the sequence \eqref{SESKDLiecoalg} is 
\begin{equation}\label{SESKDLiealg}
\mfg\,\twoheadleftarrow\,\mathit{Lie}(s^{-1}V^\ast) \,\hookleftarrow\,s^{-2}(qR)^\perp\,.
\end{equation}
The desired proposition follows from the fact that $\mathit{Lie}(s^{-1}V^\ast)/(s^{-2}(qR)^\perp)$ is universal among the Lie quotient algebras $\mfg$ of the free Lie algebra $\mathit{Lie}(s^{-1}V^\ast)$ such that the composite \eqref{SESKDLiealg} is zero.
\end{proof}

Since the free commutative algebra $\Sym(V)$ can be identified with $T(V)/([V,\,V])$, the commutative QLC algebra $A$ can be viewed as a QLC algebra in the sense of Section \ref{secQLC}: $A\cong T(V)/p^{-1}(R)$, where $p\,:\,T(V) \twoheadrightarrow \Sym(V)$ is the canonical weight preserving projection. Similarly, $qA^+\defeq k\oplus qA$, the unitalization of $qA$, can be viewed as a quadratic associative algebra. Let $(qA^+)^\dual$ denote the Koszul dual coalgebra of $A$ viewed as an {\it associative} QLC algebra (see Subsection \ref{subsecKDCCA}). 

Recall from \cite{Michaelis80}, the universal coenveloping coalgebra functor $\U^c\,:\, \GrCoLAlg_k \rightarrow \GrCoAlg_k$ is right adjoint to  the functor ${\mathcal Lie}^{c}$. For every $\mfG\in \GrCoLAlg_k$, there is a natural Lie coalgebra morphism $\pi\,:\,{\mathcal Lie}^{c}(\U^c\mfG) \rightarrow \mfG$ corresponding to the identity map on $\U^c\mfG$.

\begin{example}
Let $W$ be a graded vector space over $k$. The universal coenveloping coalgebra of the cofree Lie coalgebra $L^c(W)\defeq\mathit{Lie}^c(W)$ is $T^c(W)\defeq k\oplus\mathit{Ass}^c(W)$, where $\mathit{Ass}^c$ is the linear dual of $\mathit{Ass}$. Moreover, the natural morphism $\pi\,:\,{\mathcal Lie}^{c}(T^c(W)) \rightarrow L^c(W)$ is formally dual to the canonical inclusion from the free Lie algebra to the tensor algebra, i.e., the inclusion of primitive elements into the tensor algebra equipped with the unshuffle comultiplication. In particular, in weight $2$, it is given by the canonical projection from $W\otimes W$ to $\wedge^2 W$ (see \cite{SchSta85}).
\end{example}

The following theorem explains the relation between the Koszul dual Lie coalgebra $(qA)^\dual$ and the Koszul dual (coassociative) coalgebra $(qA^+)^\dual$.

\begin{theorem}\label{ThmKDLiecoalgKDcoalg}
There is a natural isomorphism
$$\U^c((qA)^\dual)\,\cong\,(qA^+)^\dual\,.$$
\end{theorem}

\begin{proof}
First of all, let us construct a canonical Lie coalgebra morphism from ${\mathcal Lie}^{c}((qA^+)^\dual)$ to $(qA)^\dual$. By the definition of $(qA^+)^\dual$ in \cite{LV}, it is universal among the subcoalgebras $C$ of the cofree coalgebra $T^c(sV)$ such that the composite 
\begin{equation}\label{SESKDcoalg}
C\,\hookrightarrow\,T^c(sV) \,\twoheadrightarrow\,(sV)^{\otimes 2}/(s^2qR\oplus s^2[V,\,V])
\end{equation}
is zero, where $\Sym^2(V)\hookrightarrow V\otimes V$ is the symmetrization. Let $i_{\mathit{Ass}}$ (resp., $i_{\mathit{Lie}}$) be the canonical monomorphism from $(qA^+)^\dual$ to $T^c(sV)$ (resp., from $(qA)^\dual$ to $L^c(sV)$), and $\pi\,:\,{\mathcal Lie}^{c}(T^c(sV))\rightarrow L^c(sV)$ be the natural Lie coalgebra morphism induced by $(T^c(sV))\cong \U^c(L^c(sV))$. Since $\wedge^2(sV)\cong (sV)^{\otimes 2}/s^2[V,\,V]$, the following composition 
$${\mathcal Lie}^{c}((qA^+)^\dual)\,\xrightarrow{{\mathcal Lie}^{c}(i_{\mathit{Ass}})}\,{\mathcal Lie}^{c}(T^c(sV))\,\xrightarrow{\pi}\, L^c(sV)\,\twoheadrightarrow\,\wedge^2(sV)/s^2qR $$
is zero. Therefore, there is a unique Lie coalgebra morphism, denoted by $j$, making the following diagram commutes:
\begin{equation}\label{KDLiecoalgKDcoalgCD1}
\begin{diagram}[height=0.8cm]
(qA)^\dual & & \rTo^{i_{\mathit{Lie}}} & & L^c(sV)\\
& \luTo_{j} & & \ruTo_{\ \pi\circ {\mathcal Lie}^{c}(i_{\mathit{Ass}})} &\\
& & {\mathcal Lie}^{c}((qA^+)^\dual)\,. & &
\end{diagram}
\end{equation}

Let $C$ be a graded coalgebra over $k$, and suppose one has a Lie coalgebra morphism $f\,:\,{\mathcal Lie}^{c}(C) \rightarrow (qA)^\dual $. By the universal property of the cofree Lie coalgebra $L^c(sV)$, the Lie coalgebra morphism $i_{\mathit{Lie}}\circ f$ is determined by a linear map $g$ from ${\mathcal Lie}^{c}(C)$ to $sV$. Since the underlying graded vector space of ${\mathcal Lie}^{c}(C)$ is $C$, the very same $g$ determines a coalgebra morphism from $C$ to $T^c(sV)$, denoted by $\tilde{g}$. And there is a commutative diagram:
\begin{equation}\label{KDLiecoalgKDcoalgCD2}
\begin{diagram}[height=0.8cm]
{\mathcal Lie}^{c}(C) & & \rTo^{i_{\mathit{Lie}}\circ f} & & L^c(sV)\\
& \rdTo_{{\mathcal Lie}^{c}(\tilde{g})\ } & & \ruTo_{\pi} &\\
& & {\mathcal Lie}^{c}(T^c(sV))\,. & &
\end{diagram}
\end{equation}

Since the composite
$${\mathcal Lie}^{c}(C)\,\xrightarrow{i_{\mathit{Lie}}\circ f}\,L^c(sV)\,\twoheadrightarrow\,\wedge^2(sV)/s^2qR $$
is zero, for every $c\in C$, 
$$(g\otimes g)\circ (\Delta(c)-\tau_{(12)}\Delta(c))\,=\,(g\otimes g)\circ\ ]c[\ \,\in\,s^2qR\,.$$
Hence, for every $c\in C$, 
$$(g\otimes g)\circ \Delta(c)\in s^2qR\oplus s^2[V,\,V]\,,$$
which implies that there exist a unique coalgebra morphism $h\,:\,C \rightarrow (qA^+)^\dual$ making the following diagram commutes:
\begin{equation}\label{KDLiecoalgKDcoalgCD3}
\begin{diagram}[height=0.7cm]
& & (qA^+)^\dual & &\\
& \ruTo^{h} & & \rdTo^{i_{\mathit{Ass}}}&\\
C & & \rTo^{\tilde{g}} & & T^c(sV)\,.
\end{diagram}
\end{equation}
Combing the commutative diagrams \eqref{KDLiecoalgKDcoalgCD1}, \eqref{KDLiecoalgKDcoalgCD2}, and \eqref{KDLiecoalgKDcoalgCD3}, we obtain
$$i_{\mathit{Lie}}\circ f\,=\,\pi\circ {\mathcal Lie}^{c}(i_{\mathit{Ass}})\circ {\mathcal Lie}^{c}(h)\,=\,i_{\mathit{Lie}}\circ j \circ {\mathcal Lie}^{c}(h)\,.$$
Since $i_{\mathit{Lie}}$ is a monomorphism, one has
$$f\,=\,j \circ {\mathcal Lie}^{c}(h)\,.$$
Finally, suppose there exist another coalgebra morphism $h'\,:\,C \rightarrow (qA^+)^\dual$ such that $f=j \circ {\mathcal Lie}^{c}(h')\,.$ It can be shown that $h'$ could make diagram \eqref{KDLiecoalgKDcoalgCD3} commutes as well. It follows that $h=h'$, which concludes the proof of the desired theorem.
\end{proof}

\begin{remark}
Let $\GrCHAlg_k$ denote the category of graded commutative Hopf algebras over $k$. It is known that the universal coenveloping coalgebra functor $\U^c\,:\, \GrCoLAlg_k \rightarrow \GrCoAlg_k$ can be factorized as
$$\GrCoLAlg_k \,\xrightarrow{\U^c_{H}}\,\GrCHAlg_k\,\xrightarrow{F}\,\GrCoAlg_k\,,$$
where $F\,:\,\GrCHAlg_k\rightarrow \GrCoAlg_k$ is the forgetful functor (see \cite{Michaelis80} for the construction of $\U^c_{H}$). Hence, $(qA^+)^\dual$ carries the structure of a graded commutative Hopf algebra.
\end{remark}


\begin{thebibliography}{}

\bibitem{BKR}
Yu. Berest, G. Khachatryan and A. Ramadoss, {\it Derived representation schemes and cyclic homology}, Adv. Math. \textbf{245} (2013), 625--689.

\bibitem{BW20}
M. K. Brown and M. E. Walker, {\it A Chern-Weil formula for the Chern character of a perfect curved module}, J. Noncommut. Geom. \textbf{14} (2020), no. 2, 709--772.

\bibitem{CE56}
H. Cartan and S. Eilenberg, {\it Homological algebra}, Princeton University Press, Princeton, N. J., 1956.

\bibitem{Connes85}
A. Connes, {\it Noncommutative differential geometry}, Inst. Hautes \'{E}tudes Sci. Publ. Math. No. 62 (1985), 257--360.

\bibitem{FT85}
B. L. Feigin and B. L. Tsygan, {\it Additive K-Theory and crystalline cohomology}, Funct. Anal. Appl. \textbf{19} (1985), no. 2, 124--132.

\bibitem{FT87}
B. L. Feigin and B. L. Tsygan, {\it Cyclic homology of algebras with quadratic relations, universal enveloping algebras and group algebras}, Lecture Notes in Math. \textbf{1289}, Springer, Berlin, 1987, 210--239.

\bibitem{GK94}
V. Ginzburg and M. Kapranov, {\it Koszul duality for operads}, Duke Math. J. \textbf{76} (1994), no. 1, 203--272.

\bibitem{Good85}
T. G. Goodwillie, {\it Cyclic homology, derivations, and the free loopspace}, Topology \textbf{24} (1985), 187--215.

\bibitem{HM12}
J. Hirsh and J. Mill\`{e}s, {\it Curved Koszul duality theory}, Math. Ann. \textbf{354} (2012), no. 4, 1465--1520.

\bibitem{Hoch45}
G. Hochschild, {\it On the cohomology groups of an associative algebra}, Ann. of Math. \textbf{46} (1945), 58--67.

\bibitem{Idrissi18}
N. Idrissi, {\it Curved Koszul duality of algebras over unital versions of binary operads}, (2018), \texttt{arXiv:1805.01853}.

\bibitem{Jones87}
J. D. S. Jones, {\it Cyclic homology and equivariant homology}, Invent. Math. \textbf{87} (1987), 403--423.

\bibitem{Kassel88}
C. Kassel, {\it L'homologie cyclique des alg\`ebres enveloppantes}, Invent. Math. \textbf{91} (1988), no. 2, 221--251.

\bibitem{Kassel92}
C. Kassel, {\it Cyclic homology of differential operators, the Virasoro algebra and a $q$-analogue}, Comm. Math. Phys. \textbf{146} (1992), no. 2, 343--356.

\bibitem{Kassel06}
C. Kassel, {\it Homology and cohomology of associative algebras. A concise introduction to cyclic homology}, (2006), {\tt https://cel.archives-ouvertes.fr/cel-00119891/document}.

\bibitem{L}
J.-L. Loday, {\it Cyclic homology}, Grundl. Math. Wiss. \textbf{301}, 2nd Ed., Springer-Verlag, Berlin, 1998.

\bibitem{LQ}
J.-L. Loday and D. Quillen, {\it Cyclic homology and the Lie algebra homology of matrices}, Comment. Math. Helv. \textbf{59} (1984), no. 4, 569--591.

\bibitem{LV}
J.-L. Loday and B. Vallette, {\it Algebraic Operads}, Grundl. Math. Wiss. \textbf{346}, Springer, Heidelberg, 2012.

\bibitem{Lyuba13}
V. V. Lyubashenko, {\it Bar and cobar constructions for curved algebras and coalgebras}, Mat. Stud. \textbf{40} (2013), no. 2, 115--131.

\bibitem{Michaelis80}
W. Michaelis, {\it Lie coalgebras}, Adv. Math. \textbf{38} (1980), 1--54.

\bibitem{JBMilles12}
J. Mill\`{e}s, {\it The Koszul complex is the cotangent complex}, Int. Math. Res. Not. IMRN (2012), no. 3, 607--650.

\bibitem{Priddy}
S. B. Priddy, {\it Koszul resolutions}, Trans. Amer. Math. Soc. \textbf{152} (1970), 39--60.

\bibitem{PP05}
A. Polishchuk and L. Positselski, {\it Quadratic algebras}, University Lecture Series \textbf{37}, American Mathematical Society, Providence, RI, 2005.

\bibitem{Pos93}
L. E. Positselski, {\it Nonhomogeneous quadratic duality and curvature}, Funct. Anal. Appl. \textbf{27} (1993), no. 3, 197--204.

\bibitem{Pos11}
L. Positselski, {\it Two kinds of derived categories, Koszul duality, and comodule-contramodule correspondence}, Mem. Amer. Math. Soc. \textbf{212} (2011), no. 996.

\bibitem{Qui67}
D. Quillen, {\it Homotopical Algebra}, Lecture Notes in Math. \textbf{43}, Springer-Verlag, Berlin, 1967.

\bibitem{Qui69}
D. Quillen, \textit{Rational homotopy theory}, Ann. Math. \textbf{90} (1969), 205-295.

\bibitem{Qui89}
D. Quillen, {\it Algebra cochains and cyclic homology}, Inst. Hautes Etudes Sci. Publ. Math. \textbf{68} (1989), 139--174.

\bibitem{SchSta85}
M. Schlessinger and J. Stasheff, {\it The Lie algebra structure of tangent cohomology and deformation theory}, J. Pure Appl. Algebra \textbf{38} (1985), 313--322.

\bibitem{Sridharan61}
R. Sridharan, {\it Filtered algebras and representations of Lie algebras}, Trans. Amer. Math. Soc. \textbf{100} (1961), 530--550.

\bibitem{Tsygan83}
B. L. Tsygan, {\it Homology of matrix Lie algebras over rings and the Hochschild homology}, Uspekhi Mat. Nauk \textbf{38} (1983), no. 2, 217--218 (Russian).

\bibitem{Weibel94}
C. A. Weibel, {\it An introduction to homological algebra}, Cambridge Studies in Advanced Mathematics, \textbf{38}, Cambridge University Press, Cambridge, 1994.
\end{thebibliography}
\end{document}